\documentclass[11pt,reqno]{amsart}

\usepackage{theoremref}
\usepackage{mathrsfs}
\usepackage{hyperref}

\theoremstyle{plain}

\usepackage{kotex}
\usepackage{amsmath,amssymb,amscd,color}

\usepackage{tikz}
\usepackage{times}
\usepackage{color}
\usepackage{amsmath}
\usepackage{amssymb}
\usepackage{amsthm}
\usepackage{enumerate}
\usepackage{amsbsy}
\usepackage{amsfonts}
\newtheorem{theo}{Theorem}[section]

\newtheorem{prop}[theo]{Proposition}

\newtheorem{lemm}[theo]{Lemma}
\newtheorem{rem}[theo]{Remark}

\newcommand{\al}{\alpha}

\newcommand{\Om}{\Omega}

\newcommand{\si}{\sigma}

\newcommand{\ep}{\epsilon }

\newcommand{\pa}{\partial}

\newcommand{\na}{\nabla}

\begin{document}
\baselineskip=18pt

\title[]{
Global well-posedness and time-decay estimates of the Navier-Stokes equations in exterior domains for critical data
}

\
\author[T.Chang]{Tongkeun Chang}
\address{Department of Mathematics, Yonsei University \\
Seoul, 136-701, South Korea}
\email{chang7357@yonsei.ac.kr}

\author[B.J.Jin]{Bum Ja Jin}
\address{Department of Mathematics, Mokpo National University, Muan-gun 534-729,  South Korea }
\email{bumjajin@mnu.ac.kr}

\thanks{Tongkeun Chang\\
email: chang7357@yonsei.ac.kr\\
Department of Mathematics, Yonsei University Seoul, 136-701, South Korea\\
Bum Ja Jin\\
email: bumjajin@mnu.ac.kr\\
Department of Mathematics, Mokpo National University, Muan-gun 534-729,  South Korea
}

\begin{abstract} 
  
This paper addresses the initial boundary value problem for the Navier–Stokes equations in a smooth exterior domain. We explore the global strong solvability for initial data that is small in a certain critical space larger than $L^n_\sigma$. In particular, this paper extends the results of M. Cannone, F. Planchon, and M. Schonbek \cite{CPS}, which were established for the half-space, to the exterior domain problem.


\end{abstract}

\maketitle

\section{\bf Introduction}
\setcounter{equation}{0}
In this paper, we consider the initial-boundary value problem for Navier-Stokes equations in  exterior domain $\Om$ in ${\mathbb R}^n$, $n\geq 3$:
\begin{align}
\label{e1}
\left\{\begin{array}{l}\vspace{2mm}
\partial_t{u}-\Delta {u}+\nabla p=-(u\cdot \nabla)u(= -{\rm div} ({u} \otimes {u})) \,\, \mbox{ in } \,\, \Omega\times (0,\infty),\\
\vspace{2mm}
\mbox{\rm div}\, {u}=0 \,\, \mbox{ in }\,\, \Omega\times (0,\infty),\\
\vspace{2mm}
{u}|_{\pa\Om}=0\,\, \mbox{ for } \,\,  t>0,  
\\
\vspace{2mm}
 \lim_{|x|\rightarrow \infty}{u}(x,t)=  \,\, 0\mbox{ for }t>0,\\
{u}|_{t=0}={u}_0\,\, \mbox{ in } \,\, \Omega.
\end{array}\right.
\end{align}
 Here, ${u} = (u_1, u_2, \cdots, u_n)$ and $p$ denote unknown velocity and pressure, respectively, while  ${u}_0$ is a given initial velocity.

Let ${\mathbb P}$ be the continuous projection operator  onto solenoidal vector fields  
and   $A=-{\mathbb P}\Delta$. 
The classical system \eqref{e1} can be reduced into  an  integral equation
 \begin{equation}
 \label{e1new}
 u(t)=e^{-tA}u_0-\int^t_0e^{-tA(t-\tau)}{\mathbb P}(u\cdot \nabla u)(\tau) d\tau.
 \end{equation}
 A solution
 in $L^q(0,T;L^p_{\sigma}(\Om))$  
 satisfying \eqref{e1} in distributional sense is called a very weak solution and a solution
  in $C(0,T;X)$  
 satisfying \eqref{e1new}  is called a mild solution in $X$.
Here,  $X$ denotes some Banach space with ${\mathbb P}X\subseteq X$.
It is well known that the  very weak solutions, the mild solutions and the strong solutions    coming from   the same initial data coincide with each other    as far as they are in $C(0,T; L^p_\sigma), p\geq n$ (see \cite{Amann}). 
In this sense, we will not  make a distinction between  the types of solutions without further  mention.

The solvability of the Navier-Stokes equations in the whole space has been well known for the initial data in $L^p_\sigma({\mathbb R}^n)$ for any $p\geq  n$(see \cite{daveiga,calderon,fabes,kato} and references therein). The well-posedness for the data in $L^p_\sigma(\Om)$ have been studied also for  the other standard  domains such as half spaces, bounded domains and  exterior domains (see \cite{kozono0,ukai} for the half spaces, \cite{giga,miyakawa,wahl} for  bounded domains, \cite{iwashita} for the exterior domains).  In those literature, the local solvability has been studied for $u_0\in L^p_\sigma,p\geq n$, and the global solution has been obtained if $u_0$ is small in   $L^n_\sigma$.

A rough computation shows  that
\[
\|e^{-(t-\tau)A}{\mathbb P}(u\cdot \nabla )u(\tau)\|_{L^p}\leq c(t-\tau)^{-\frac{1}{2}-\frac{n}{2p}}\|u(\tau)\|_{L^p}^2 ,
\]
and the right hand side is integrable in $\tau$ variable for  $p>n$ but not for  $p=n$.  
In this sense, the space  $L^n_\sigma$ is critical. 
 The space $L^n_\sigma$ is  invariant under the scaling
$
\lambda v(\lambda x
)$.
In general,   the  scaling invariant spaces are  called critical spaces.

 If the data $u_0$ is in the critical space $L^n_\sigma$, then we could obtain a closed form
\[
\|u\|_{K}\leq c\big( \|u_0\|_{L^n}+\|u\|_{K}^2 \big),
\]
where $\|u\|_{K}=\sup_{0\leq t<T}\|u(t)\|_{L^n}+\sup_{0\leq t<T}t^{\frac{1}{2}-\frac{n}{2p}}\|u(t)\|_{L^p}$, $p>n$.
With this argument, T. Kato \cite{kato} showed the global-in-time solvability for the data small in $L^n_\sigma({\mathbb R}^n)$.

The same argument can be applicable for any domain as far as $L^r-L^p$ estimates for the Stokes semigroup operators have been obtained. Also, it can be extended to any critical space $X$ if we could derive the estimates
\[
\|e^{-tA}u_0\|_{K}\leq c\|u_0\|_{X}.
\]

M. Cannone \cite{cannone} showed the global solvability for the data small in the homogeneous Besov space $\dot{B}^{-1+\frac{n}{p}}_{p,\infty,\sigma}({\mathbb R}^n), p>n$.
Observe that $|x|^{-1}\in \dot{B}^{-1+\frac{n}{p}}_{p,\infty}({\mathbb R}^n), p>n$ but $|x|^{-1}\notin L^n({\mathbb R}^n)$.
(See also \cite{tataru} for the data in even the weaker space ${\rm BMO}^{-1}({\mathbb R}^n)$).

The same result could be expected for the domains with non-empty boundaries. However, the non-local properties of the Besov norm make it difficult to show 
\[
\|e^{-tA}u_0\|_{K}\leq c\|u_0\|_{\dot{B}^{-1+\frac{n}{p}}_{p,\infty,\si}(\Om)} \mbox{ when }\Om\neq {\mathbb R}^n.
\]
Fortunately, the Stokes flow in the half-space has an explicit solution representation consisting of the compositions of the heat semigroup operators, Riesz operators and their convolutions. This helps M. Cannone, F. Planchon and M. Schonbek \cite{CPS} to prove the global strong solvability for the data small in $\dot{B}^{-1+\frac{n}{p}}_{p,\infty,\sigma}({\mathbb R}^n_+)$, $p>n$. 

Poincar\'{e}'s inequality holds for the functions in bounded domains. This leads to an exponential decay estimate of the Stokes operator in a bounded domain and helps H. Amann \cite{Amann} to show the global strong solvability for the data small in little Nikolskii space ${n}^{-1+\frac{n}{p}}_{p,0,\sigma}(\Om), p>n$.
Here ${n}^{-1+\frac{n}{p}}_{p,0,\sigma}$ is little Nikolskii space (see \cite{Amann}) satisfying the property 
\[
L^n_\sigma(\Om)\subset {n}^{-1+\frac{n}{p}}_{p,0,\sigma}(\Om)\subset B^{-1+\frac{n}{p}}_{p,\infty,\sigma}(\Om), \,p>n.
\]
See also \cite{FS,FS1,FS2} for the data in $B^{-1+\frac{n}{p}}_{p,q,\sigma}(\Om), \, \frac{n}{p}+\frac{2}{q}\leq 1, \,n<p<\infty, \, 2<q\leq \infty$.  
%
%
%
In \cite{Amann}, the local strong solvability for $u_0\in {n}^{-1+\frac{n}{p}}_{p,0,\sigma} (\Om)$ has been proved for any standard domains, including exterior domains. However, the corresponding global solvability was established for all such domains except exterior domains. Similarly, M.-H. Ri, P. Zhang and Z. Zhang \cite{MZZ} proved local and global solvability for data in another critical space, $B^0_{n,\infty,\sigma}(\Om)$, but their global result also does not cover exterior domains. 

Unlike bounded domains---where homogeneous and inhomogeneous spaces coincide, offering significant analytical advantages---such benefits are absent in the exterior domain problem. Consequently, to the best of our knowledge, there has been no global-in-time solvability result for the Navier-Stokes equations in homogeneous Besov spaces on exterior domains.

Therefore, the primary aim of this paper is to overcome this difficulty and extend the global strong solvability results of M. Cannone, F. Planchon, and M. Schonbek \cite{CPS} to the exterior domain problem.

{\color{red}{
On the other hand, global solvability for exterior domains with rough data has been successfully established in other larger scaling-invariant spaces, most notably the weak Lebesgue spaces $L^{n,\infty}$ (or Lorentz spaces $\mathbb{L}(n,\infty)$) by H. Kozono and M. Yamazaki \cite{KozonoYamazaki1998}, M. Yamazaki \cite{Yamazaki2000}, and P. Maremonti \cite{maremonti2017}. Regarding this functional setting, it is important to clarify the relationship between our working space $\dot{b}^{-1+\frac{n}{p}}_{p,q}$ and the weak Lebesgue space $L^{n,\infty}$. We emphasize that these two spaces do not generally possess an inclusion relationship. While $L^{n,\infty}$ is well-suited for capturing singular functions like $|x|^{-1}$, such functions do not fall into $\dot{b}^{-1+\frac{n}{p}}_{p,q}$ due to the nature of the closure of $C^\infty_{0,\sigma}$. Conversely, for $p > n$, our space $\dot{b}^{-1+\frac{n}{p}}_{p,q}$ handles highly oscillatory distributions owing to its negative regularity, which $L^{n,\infty}$ fails to cover. In this sense, our approach provides a parallel and independent framework that treats a fundamentally different type of rough initial data for the exterior domain problem.
}
}

To establish our results in this functional setting, we utilize fundamental tools such as extension operators, cut-off functions, local Poincar\'{e} inequalities, and Sobolev and Besov embeddings \cite{danchin, Gau}. We explicitly establish in Section \ref{notation} (see Lemma \ref{lemma2.4}) that the Sobolev type embedding $\dot{H}^{1}_{p}(\Om) \subset L^{\frac{np}{n-p}}(\Om)$, $p<n$ still holds for exterior domains. Relying on this lemma, the subsequent propositions (Propositions \ref{Proposition2.7} and \ref{Proposition2.8}) deduce the crucial interpolation properties:
\[
(L^p (\Om), \dot{H}^{1}_{p} (\Om) )_{\al, 1}=\dot{B}^{\al}_{p,1} (\Om)\mbox{ and }(L^p (\Om), \dot{H}^{1}_{p,0} (\Om) )_{\al, 1}=\dot{B}^{\al}_{p,1,0} (\Om),\mbox{ for }p<n.
\]
These preliminary results serve as fundamental tools for our main theorems. Indeed, applying these interpolation properties to $L^p-L^r$ estimates of the Stokes operator, we can derive the Stokes estimate from $\dot{B}^\al_{p,1,0}(\Om)$ to $L^r$ for $p<n$. Making use of the duality argument, we then obtain the following estimate:
\[
\|e^{-tA}u_0\|_{K}\leq c\|u_0\|_{\dot{B}^{-\al}_{p,\infty,\si}(\Om)}, \quad p>\frac{n}{n-1}.
\]
See Section \ref{notation} for the definitions of the function spaces mentioned above and below.
The following is our estimates of the Stokes operator for the data in homogeneous Besov space. 
For $\alpha > 0$, we denote by $\dot{b}^{-\alpha}_{p,\infty,\sigma}(\Omega)$ the completion of $C^\infty_{0,\sigma}(\Omega)$ in $\dot{B}^{-\alpha}_{p,\infty,\sigma}(\Omega)$.

\begin{theo}
\label{theorem2}
Let $\Omega$ be a smooth exterior domain in $\mathbb{R}^n$, $n\geq 3$. 
Let $u_0 \in \dot{b}^{-\alpha}_{p,\infty,\sigma}(\Omega)$, where $\frac{n}{n-1} < p < \infty$ and $0 < \alpha < 1$.  
Then, it holds that 
\begin{align*}
 \|e^{-tA}u_0\|_{L^p_\sigma(\Omega)} &\leq ct^{-\frac{\alpha}{2}}\|u_0\|_{\dot{B}^{-\alpha}_{p,\infty,\sigma}(\Omega)},\\
 \|e^{-tA}u_0\|_{\dot{B}^{-\alpha}_{p,\infty,\sigma}(\Omega)} &\leq c\|u_0\|_{\dot{B}^{-\alpha}_{p,\infty,\sigma}(\Omega)}. 
\end{align*}
\end{theo}

%
%
%
%
The following estimates will be used for the bilinear operators $\int^t_0e^{-A(t-\tau)}{\mathbb P}(u\cdot \nabla v)(\tau) d\tau$. 

\begin{theo}
\label{bilinear1}
Let $\Omega$ be a smooth exterior domain in $\mathbb{R}^n$ with $n\geq 3$. 
Suppose that $0<\alpha<1$ and $\frac{n}{n-1} < r < p < \infty$. 
Then, for any tensor field $\mathcal{F}\in [L^r(\Omega)]^{n^2}$, the following estimates hold:
\begin{align*}
\|e^{-tA}{\mathbb P}\mbox{\rm div}{\mathcal F}\|_{L^p (\Omega) }
& \leq ct^{-\frac{1}{2}-\frac{n}{2}(\frac{1}{r}-\frac{1}{p})}\|{\mathcal F}\|_{L^r (\Omega) },\\
\| e^{-tA}{\mathbb P}\mbox{\rm div}{\mathcal F}\|_{\dot{B}^{-\alpha}_{p,\infty,\sigma} (\Omega) }
& \leq ct^{-\frac{1-\alpha}{2}-\frac{n}{2}(\frac{1}{r}-\frac{1}{p})}\|{\mathcal F}\|_{L^r (\Omega) }.
\end{align*}
\end{theo}

We use the contraction principle to construct the solution of the nonlinear problem \eqref{e1new}. 
Applying the estimates in Theorem \ref{theorem2} and Theorem \ref{bilinear1}, we prove the global solvability for small data in the critical space $\dot{b}^{-1+\frac{n}{p}}_{p,\infty,\sigma}(\Omega)$, which is larger than $L^n_\sigma(\Omega)$. 

Before stating our main results, let us introduce the weighted Bochner space. For a Banach space $X$, we denote by $L^q_\alpha(0,\infty;X)$ the weighted Bochner space equipped with the following norm:
\begin{align*}
\|f\|_{L^q_\alpha (0,\infty ;X)} &= \Big(\int_0^\infty  t^{q\alpha}\|f(t)\|^q_{X} dt \Big)^\frac{1}{q}, \quad 1 \leq q < \infty,\\
\|f\|_{L^\infty_\alpha (0,\infty ;X)} &= \sup_{0 < t < \infty} t^{\alpha}\|f(t)\|_{X}, \quad q=\infty.
\end{align*}
Our main results regarding the global solvability are as follows.

\begin{theo}
\label{theorem1}
Let $\Omega$ be a smooth exterior domain in $\mathbb{R}^n$ with $n\geq 3$. 
\begin{itemize}
 \item[(1)] Let $u_0 \in \dot{b}^{-1 +\frac{n}{p}}_{p,\infty,\sigma}(\Omega)$ for $n < p < \infty$. There is $\ep_0>0$ such that if $\|u_0\|_{\dot{B}^{-1 + \frac{n}{p}}_{p,\infty,\sigma}(\Omega)} < \ep_0$, then the equations \eqref{e1} has a unique solution $u$ in $L^\infty_{\frac{1}{2} -\frac{n}{2p}}(0, \infty ; L^p_\sigma(\Omega)) \cap L^\infty(0,\infty;\dot{B}^{-1+\frac{n}{p}}_{p,\infty,\sigma}(\Omega))$.  

 \item[(2)] Let $u_0 \in L^n_\sigma(\Omega)$. There is $\ep_0>0$ such that if $\|u_0\|_{\dot{B}^{-1 + \frac{n}{p_0}}_{p_0, \infty,\sigma}(\Omega)} < \ep_0$ for some $n<p_0<\infty$, then the equations \eqref{e1} has a unique strong solution $u$ in $L^\infty(0, \infty ; L^n_\sigma(\Omega)) \cap L^\infty_{\frac{1}{2} -\frac{n}{2p_0}}(0, \infty ; L^{p_0}(\Omega))$ with $t^{\frac{1}{2}}\nabla u\in L^\infty(0, \infty ; L^n(\Omega))$.  
\end{itemize}
\end{theo}

\begin{rem}
\begin{itemize}
\item[(i)] The space $\dot b^{-1 + \frac{n}{p}}_{p ,\infty,\si } (\Om)$, $p>n$ contains $L^n_\si$.

\item[(ii)] It is well known that a weak solution with ${u}\in L^\infty_{loc}(0,T;L^2_{loc}(\Om)), \ \nabla {u}\in L^2_{loc}(\Om\times (0,T))$ is smooth in the interior of the domain as far as ${u}\in L^q_{loc}(0,T;L^p_{loc}(\Om))$ with $\frac{n}{p}+\frac{2}{q}\leq 1,\  n<p\leq \infty$ (see \cite{ESS,serrin,struwe,takahashi} and the references therein). In this sense our solution is a strong one.

\item[(iii)] The space $\dot{b}^s_{p,\infty,\si}$ in our paper is comparable with the little Nikolskii space $n^s_{p,0,\si}$ in \cite{Amann}.

\item[(iv)]{\color{red}{ A distinctive aspect of our approach lies in the treatment of homogeneous functional spaces on exterior domains. Typically, working with such spaces introduces severe technical difficulties since their elements are defined modulo polynomials. For instance, overcoming the complexities of Besov and fractional spaces in unbounded or exterior domains has been a subject of intense recent study, as seen in the works on the $H^\infty$-calculus and bounded imaginary powers of the Stokes operator by Farwig and Tsuda \cite{FarwigTsuda_Linear, FarwigTsuda_Nonlinear}. However, in our present problem, by restricting our analysis to the regime $p<n$, we exploit the crucial property that elements of the homogeneous Bessel potential space $\dot{H}^1_p(\Om)$ naturally decay at infinity. This observation elegantly bypasses the complications associated with polynomial equivalence classes, enabling us to derive the required interpolation properties for exterior domains in a clean and transparent manner.
}}
\end{itemize}
\end{rem}

\subsection*{Notation}
For a vector-valued function $u = (u_1, \ldots, u_n)$ and a function space $X$ for scalar-valued functions, we write $u \in [X]^n$ to mean that each component $u_i \in X$. For simplicity, we sometimes write $u \in X$ when the distinction is clear from the context.
In Section \ref{notation}, we mainly deal with scalar-valued functions, so the notation $u \in X$ is used. In Section \ref{solenoidal} and thereafter, we focus on vector-valued functions, and the notation $u \in [X]^n$ is adopted.

\vspace{0.5cm} 

This paper is organized as follows. In Section \ref{notation}, we introduce the homogeneous function spaces and their known properties. We also show the interpolation properties of the homogeneous spaces in exterior domains, which are not found in the literature. In Section \ref{solenoidal}, we introduce the homogeneous solenoidal spaces and show their interpolation properties in exterior domains, which are required for our study. We also introduce some known estimates related to the Stokes operators. We provide the proofs of Theorem \ref{theorem2}, Theorem \ref{bilinear1}, and Theorem \ref{theorem1} in Sections \ref{auxilliaries}, \ref{section.theorem.1.4}, and \ref{external-force}--\ref{section.besov}, respectively, with Sections \ref{external-force} and \ref{section.besov} devoted to the proof of our main result, Theorem \ref{theorem1}.

\section{\bf Homogeneous  spaces and their properties}
\setcounter{equation}{0}
\label{notation}

In this section, we collect the definitions and properties of homogeneous and solenoidal function spaces that will be used throughout the paper. 
While most of these are classical (see \cite{Amann,BL,Tr,danchin,Gau}), we also establish several interpolation properties in exterior domains, which are not found in the literature.

Let $s\in \mathbb{R}$, $1\leq p,q\leq \infty$, and let $\Omega\subset \mathbb{R}^n$ be a domain. 
We denote by $H^s_p(\Omega)$ the usual (inhomogeneous) Bessel potential spaces and by $B^s_{p,q}(\Omega)$ the Besov spaces. 
The completions of $C^\infty_0(\Omega)$ with respect to these norms are denoted by $H^s_{p,0}(\Omega)$ and $B^s_{p,q,0}(\Omega)$, respectively. 
The homogeneous Sobolev and Besov spaces $\dot H^s_p(\mathbb{R}^n)$ and $\dot B^s_{p,q}(\mathbb{R}^n)$ are defined via Littlewood-Paley decompositions (see \cite{BL,danchin} for details). 
We adopt the modified definitions introduced by Bahouri, Chemin, and Danchin \cite{danchin}, which avoid quotienting by polynomials and naturally yield Banach spaces. 
For a general domain $\Omega$, the restriction and zero-extension operators are used to define $\dot H^s_p(\Omega)$ and $\dot B^s_{p,q}(\Omega)$, equipped with the following quotient norms:
\[
\|f\|_{X(\Omega)} := \inf\big\{ \|F\|_{X(\mathbb{R}^n)} : F|_\Omega = f \big\}, \quad 
X\in\{\dot H^s_p, \dot B^s_{p,q}\}.
\]
Furthermore, we define the corresponding spaces with zero boundary conditions as follows:
\[
\dot H^s_{p,0}(\Omega) := \overline{C^\infty_0(\Omega)}^{\dot H^s_p(\mathbb{R}^n)}, \qquad
\dot B^s_{p,q,0}(\Omega) := \overline{C^\infty_0(\Omega)}^{\dot B^s_{p,q}(\mathbb{R}^n)}.
\]

These spaces satisfy the standard embedding, duality, and interpolation properties (see Proposition~\ref{prop2.2}). 
For exterior domains, we establish additional characterizations in Lemma~\ref{lemma2.4}, whose detailed proofs are deferred to the Appendix.

We first recall the classical properties of the (modified) homogeneous spaces 
following \cite{danchin} for Besov spaces and \cite{Gau} for the Bessel potential spaces. 
These results will then be extended to general domains via restriction and extension operators.

Before stating these properties, let us introduce some necessary function spaces and notations. 
We define the subspace of Schwartz functions whose Fourier transforms avoid the origin as 
${\mathcal S}_0({\mathbb R}^n):=\{ f\in {\mathcal S}: \ 0\notin \mbox{\rm supp}({\mathcal F}(f))\}.$
Furthermore, we denote by $\dot{B}^s_{p,\infty,0}({\mathbb R}^n)$ the completion of ${\mathcal S}_0$ in $\dot{B}^s_{p,\infty}({\mathbb R}^n)$. 
Throughout this paper, $X'$ denotes the dual space of $X$, and $r'=\frac{r}{r-1}$ is the Hölder conjugate of $r$. 
The notations $[A,B]_\theta$ and $(A ,B)_{\theta,q}$ stand for the complex and real interpolation spaces between normed spaces $A$ and $B$, respectively. 
Finally, the notation $f \approx g$ means that $\frac{1}{c}f\leq g\leq cf$ for some generic constant $c>0$.

\begin{prop}
\label{prop2.2}
\begin{itemize}

\item[(i)] The (modified) homogeneous spaces are Banach spaces for $s\in {\mathbb R}$, $1 < p < \infty$, $1 \leq q < \infty$ with $[s<\frac{n}{p}]$ or $[s=\frac{n}{p}$ and $q=1]$.

\item[(ii)] For any $1 < p < \infty$, $1\leq q < \infty$, $s\in {\mathbb R}$, ${\mathcal S}_0$ is dense in $L^p({\mathbb R}^n)$, $\dot{H}^s_p({\mathbb R}^n)$, and $\dot{B}^s_{p,q}({\mathbb R}^n)$. 
Moreover, we have $\dot{B}^s_{p,\infty,0}({\mathbb R}^n)\subsetneq \dot{B}^s_{p,\infty}({\mathbb R}^n)$.

\item[(iii)] For $s\in {\mathbb R}$, $1 < p < \infty$, $1 \leq q < \infty$, $m\in {\mathbb N}$, and for all $u\in {\mathcal S}_0$,
\begin{align*}
\|u\|_{\dot{H}^{s}_p({\mathbb R}^n)} &\approx \|D ^m_x u\|_{\dot{H}^{s-m}_p({\mathbb R}^n)}\approx \sum_{k=1}^n\|\partial_{x_k}^mu\|_{\dot{H}^{s-m}_p({\mathbb R}^n)},\\
\|u\|_{\dot{B}^{s}_{p,q}({\mathbb R}^n)} &\approx \|D ^m_x u\|_{\dot{B}^{s-m}_{p,q}({\mathbb R}^n)}\approx \sum_{k=1}^n\|\partial_{x_k}^mu\|_{\dot{B}^{s-m}_{p,q}({\mathbb R}^n)}.
\end{align*}
By the density argument, the above norm equivalences hold for any $u\in \dot{H}^s_p({\mathbb R}^n)$ and $u\in \dot{B}^s_{p,q}({\mathbb R}^n)$. In particular, $\dot{H}^0_p({\mathbb R}^n)=L^p({\mathbb R}^n)$.

\item[(iv)] Let $p\in (1,\infty)$, $s\in (0,\frac{n}{p})$ and $r=\frac{np}{n-sp}$. Then, the following embeddings hold:
\begin{align*}
\dot{H}^s_p({\mathbb R}^n)\subset L^r({\mathbb R}^n),\quad 
L^p({\mathbb R}^n)\subset \dot{H}^{-s}_r({\mathbb R}^n),\\
\dot{B}^s_{p,q}({\mathbb R}^n)\subset L^r({\mathbb R}^n)\ \mbox{ for }q\in [1,r], \\ 
L^p({\mathbb R}^n)\subset \dot{B}^{-s}_{r,q}({\mathbb R}^n)\ \mbox{ for }q\in [r,\infty].
\end{align*}
In particular, $\dot{B}^0_{p,r}({\mathbb R}^n)\subset L^p({\mathbb R}^n)\subset \dot{B}^0_{p,s}({\mathbb R}^n)$ for $1\leq r\leq \min\{2,p\}\leq s\leq \infty$.

\item[(v)] Let $s\in {\mathbb R}$ and $1 < p < \infty$. Then, $\dot{H}^s_p({\mathbb R}^n)$ and $\dot{B}^s_{p,q}({\mathbb R}^n)$ are reflexive for $1 < q < \infty$ with $s<\frac{n}{p}$. Moreover,
\begin{align*}
(\dot{H}^{-s}_{p'}({\mathbb R}^n))' &= \dot{H}^s_p({\mathbb R}^n) \quad \mbox{ for } s<\frac{n}{p}, \\
(\dot{B}^{-s}_{p',q'}({\mathbb R}^n))' &= \dot{B}^s_{p,q}({\mathbb R}^n)\quad \mbox{ for } -\frac{n}{p'}<s<\frac{n}{p}, \,\, 1<q\leq \infty,\\
(\dot{B}^{-s}_{p',\infty,0}({\mathbb R}^n))' &= \dot{B}^s_{p,1}({\mathbb R}^n) .
\end{align*}

\item[(vi)] Let $0 < p_0, p_1 < \infty$, $1 \leq q_0, q_1 \leq \infty$, $s_0, s_1\in {\mathbb R}$, $s_0\neq s_1$. 
Let $(s,\frac{1}{p}, \frac{1}{q}):=(1-\theta)(s_0,\frac{1}{p_0},\frac{1}{q_0})+\theta (s_1,\frac{1}{p_1},\frac{1}{q_1})$ for some $\theta\in (0,1)$.
For $[s_i<\frac{n}{p_i}]$ or $[s_i=\frac{n}{p_i}$ and $q_i=1]$ ($i=0,1$), the following interpolation properties hold:
\[
[\dot{H}^{s_0}_{p_0}({\mathbb R}^n), \dot{H}^{s_1}_{p_1}({\mathbb R}^n)]_\theta=\dot{H}^{s}_{p}({\mathbb R}^n),\quad
(\dot{H}^{s_0}_{p_0}({\mathbb R}^n), \dot{H}^{s_1}_{p_1}({\mathbb R}^n))_{\theta,q}=\dot{B}^{s}_{p,q}({\mathbb R}^n).
\]

\end{itemize}
\end{prop}

{

Next we define the restriction spaces on domains

Let   $\Om$ be a domain in $ {\mathbb R}^n$. 
Let   $s\geq 0, \, 1 <  p < \infty, \,\, 1 \leq  q \leq \infty$.  
The spaces $\dot H^s_p (\Om)$ and    $ \dot B^s_{p,q}(\Om)$  are  defined by the restrictions of $\dot H^s_p ({\mathbb R}^n) $ and    $ \dot B^s_{p,q}({\mathbb R}^n)$ over $\Om$  normed with  
  \begin{align} \label{norminOm}
  \| f \|_{X(\Om)} : = \inf \{  \|F\|_{  X({\mathbb R}^n)} \,:\,  f \in  X({\mathbb R}^n), \,\,  F|_{\Om} =f\},
  \end{align}
  where $X({\mathbb R}^n) \in \{ \dot{H}^s_{p}({\mathbb R}^n), \,\,        \dot{B}^s_{p,q}({\mathbb R}^n) \}$ and  $X(\Om) \in \{\dot{ H}^s_{p}(\Om), \,\,      \dot{ B}^s_{p,q}(\Om)\}$.
  Here the curly braces simply list the possible choices of spaces.

  Define 
  \begin{align*}
  \dot H^s_{p,0} (\Om):&=\{u\in \dot{H}^s_p({\mathbb R}^n): \ {\rm supp} \,\, u\subset \bar{\Om}\},\\
   \dot B^s_{p,q,0}(\Om):&=\{u\in \dot{B}^s_p({\mathbb R}^n): \ {\rm supp}\,\, u\subset \bar{\Om}\}
   \end{align*}   normed with
  \[
  \|u\|_{ \dot H^s_{p,0} (\Om)}:=\|u\|_{\dot{H}^s_p({\mathbb R}^n)},\quad \|u\|_{ \dot B^s_{p,q,0} (\Om)}:=\|u\|_{\dot{B}^s_{p,q}({\mathbb R}^n)}.
  \]
It is trivial that  $\dot{H}^0_p(\Om) = \dot{H}^0_{p,0}(\Om)=L^p(\Om)$ and ${\mathcal S}_0(\overline{\Om})$ is  dense in $X(\Om)$, where ${\mathcal S}_0(\overline{\Om})$ is the restriction of ${\mathcal S}_0$ to $\Om$.
}

Define the spaces
\[
\dot{H}^{-s}_p(\Om):=(\dot{H}^s_{p',0}(\Om))',
\ \dot{B}^{-s}_{p,q}(\Om):=(\dot{B}^{s}_{p',q',0}(\Om))',\]
and
\[
\dot{H}^{-s}_{p,0}(\Om):=(\dot{H}^s_{p'}(\Om))', \quad  \dot{B}^{-s}_{p,q,0}(\Om):=(\dot{B}^{s}_{p',q'}(\Om))'.\]

By the definition of restrictions  the following embedding holds:
\begin{lemm}
\label{lemma2.1}
Let $\Om$ be a domain in ${\mathbb R}^n$. Let $ 1 \leq p \leq \infty, \,\,  s\in (0,\frac{n}{p}), \,\, r=\frac{np}{n-sp}$. 
\begin{align*}
&\dot{H}^{s}_p(\Om) \subset L^{r}(\Om), \quad
\dot{B}^{s}_{p,q}(\Om)\subset L^{r}(\Om))\ \mbox{ for }\  q\in [1,r],\\
&\dot{B}^0_{p,r}(\Om)\subset L^p(\Om)\subset \dot{B}^0_{p,s,0}(\Om)\ \mbox{ for }\ 1\leq r\leq \min\{2,p\}\leq s\leq \infty.
\end{align*}
\end{lemm}

By duality we obtain the following embeddings
\begin{lemm}
\label{lemma2.2}
Let $\Om$ be a domain in ${\mathbb R}^n$. 
Let $p\in [1,\infty], s\in (0,\frac{n}{p}),r=\frac{np}{n-sp}$. 
Then,
\begin{align*}
L^p(\Om)\subset \dot{H}^{-s}_{r,0}(\Om),\quad
L^p(\Om)\subset\dot{B}^{-s}_{r,q,0}(\Om), \quad  q\in [r,\infty].
\end{align*}
\end{lemm}

%
%
%

Next, having established the embeddings and duality properties (Lemmas~\ref{lemma2.1}–\ref{lemma2.2}), 
we now turn to the boundedness of restriction and extension operators, which will be crucial 
for extending the classical results to exterior domains.

Let    $R_\Om $ be the restriction operator for  the  function spaces to  $\Om$ defined by  $R_\Om F=F|_{\Om}$.   
By the definition of restriction,  it is obvious that $R_\Om $ is bounded operator as follows:
\begin{lemm}
\label{lemma2.3}
Let    $1< p<\infty$, $1\leq q\leq \infty$ and $s>0$. 
Then,

\[
R_\Omega \mbox{ is bounded from }
\{L^p(\mathbb{R}^n), \dot{H}^s_p(\mathbb{R}^n), \dot{B}^s_{p,q}(\mathbb{R}^n)\}
\mbox{ to }
\{L^p(\Omega), \dot{H}^s_p(\Omega), \dot{B}^s_{p,q}(\Omega)\}.
\]

%
%

\end{lemm}

Let ${\mathcal E}_\Om$ be the zero extension   operator  for the functions   in $\Om$.   
By the definition, it  is obvious that 
 ${\mathcal E}_\Om$ is bounded operator as follows: 
\begin{lemm}
\label{lemma2.5}
Let    $1< p<\infty$, $1\leq q\leq \infty$ and $s>0$. Then 
 ${\mathcal E}_\Om$ is bounded from $ X_0(\Om):=\{L^p(\Om) , \dot{H}^s_{p,0}(\Om), \dot{B}^s_{p,q,0}(\Om)\}$ to $ X({\mathbb R}^n):=\{L^p({\mathbb R}^n) , \dot{H}^s_{p}({\mathbb R}^n), \dot{B}^s_{p,q}({\mathbb R}^n)\}$ with
 \[
 \|{\mathcal E}_\Om  u\|_{X({\mathbb R}^n)}=\|u\|_{X_0(\Om)}\,\, \mbox{ for all } \,\, u\in X_0(\Om).
 \] 
\end{lemm}

The  properties  as in Proposition \ref{prop2.2}  also holds  for the half space and special Lipschitz domains 
(see \cite{danchin2, Gau} for the half spaces and \cite{Gau2} for special Lipschitz domains, where   Special Lipschitz doamins means a  domain whose boundary is a graph of  a Lipschitz function on ${\mathbb R}^{n-1}$ ).

The above properties are expected to hold also for exterior domains. 
However, we were not able to find such results in the literature. 
Since these properties will be required in the subsequent sections, 
we establish them in Lemma \ref{lemma2.4}, Proposition \ref{Proposition2.7}, and Proposition \ref{Proposition2.8} below, 
and provide their proofs in Appendix \ref{appendix1}. 
(We believe that more general theories, such as those in Proposition \ref{prop2.2}, 
should also extend to exterior domains, although here we treat only simple cases.)

%
We deliberately define $\dot{H}_{p}^{1}(\Omega)$ via restriction from the whole space to seamlessly inherit interpolation properties.
To clarify its relation to standard functional settings, Lemma \ref{lemma2.4} below shows that this restriction-based space coincides exactly with the classical spaces (e.g., Galdi \cite[Chapter II]{galdi} and \cite{CM}), ensuring that the standard Sobolev inequality naturally holds.

\begin{lemm}[Characterization of first-order spaces]
\label{lemma2.4}
Let $\Omega$ be an exterior domain with smooth boundaries.
\begin{itemize}
\item[(i)] Let $1<p<n$. Then, the space $\dot{H}_{p}^{1}(\Omega)$ can be equivalently characterized as the following set:
\begin{equation}
\dot{H}_{p}^{1}(\Omega) \simeq \Big\{f\in L^{\frac{np}{n-p}}(\Omega) \ \Big| \ \nabla f \in [L^{p}(\Omega)]^n\Big\} 
\end{equation}
with the equivalent norm $\|u\|_{\dot{H}_{p}^{1}(\Omega)}\approx\|\nabla u\|_{L^{p}(\Omega)}+\|u\|_{L^{\frac{np}{n-p}}(\Omega)}.$
 
\item[(ii)] Let $1\leq p<n$. Then $C^\infty_0(\Om)$ is dense in $\dot{H}^1_{p,0}(\Om)$.
Furthermore, the space $\dot{H}_{p,0}^{1}(\Omega)$ is characterized by:
$$\dot{H}_{p,0}^{1}(\Omega) \simeq \Big\{u\in L^{\frac{np}{n-p}}(\Omega) \ \Big| \ \nabla u\in [L^{p}(\Omega)]^n, \ u\big|_{\partial\Omega}=0\Big\}.$$
\end{itemize}
\end{lemm}

\begin{proof}
The rigorous proof of this lemma is provided in Appendix \ref{app:char}.
\end{proof}

\begin{prop}[Extension and restriction operators]
\label{Proposition2.7}
 Let $\Omega$ be an exterior domain and $1<p<n$.
 \begin{itemize}
 \item[(i)] There is an extension operator 
 $E_{\Omega}:(C_{0}^{\infty}(\Omega))^{\prime}\rightarrow(C_{0}^{\infty}(\mathbb{R}^{n}))^{\prime}$ such that $E_{\Omega}$ is bounded from 
 $\{L^{p}(\Omega), \dot{H}_{p}^{1}(\Omega)\}$ to $\{L^{p}(\mathbb{R}^{n}), \dot{H}_{p}^{1}(\mathbb{R}^{n})\}$.
 
 \item[(ii)] There is a bounded operator $\mathcal{R}_{\Omega} : \{L^{p}(\mathbb{R}^{n}), \dot{H}_{p}^{1}(\mathbb{R}^{n})\} \rightarrow \{L^{p}(\Omega), \dot{H}_{p,0}^{1}(\Omega)\}$.

%
%

\end{itemize}
\end{prop}
\begin{proof}
The proof is given in Appendix \ref{app:oper}.
\end{proof}

\begin{prop}[ Interpolation identities]
\label{Proposition2.8}

Let $\Omega$ be an exterior domain and $0<s<1$. The spaces defined in Section 2 satisfy the following interpolation properties:

\begin{itemize}
\item[(i)] For $1<p<n$ and $1\le q\le\infty$,
$$[L^{p}(\Omega), \dot{H}_{p}^{1}(\Omega)]_{s} = \dot{H}_{p}^{s}(\Omega), \quad (L^{p}(\Omega), \dot{H}_{p}^{1}(\Omega))_{s,q} = \dot{B}_{p,q}^{s}(\Omega).$$
\item[(ii)] For $1<p<n$ and $1\le q<\infty$,
$$[L^{p}(\Omega), \dot{H}_{p,0}^{1}(\Omega)]_{s} = \dot{H}_{p,0}^{s}(\Omega), \quad (L^{p}(\Omega), \dot{H}_{p,0}^{1}(\Omega))_{s,q} = \dot{B}_{p,q,0}^{s}(\Omega).$$
%
%
%
%
%
%
%
%
%
%
%
%
%
%
\end{itemize}
\end{prop}

\begin{proof}
The detailed proof can be found in Appendix \ref{app:interp}.
\end{proof}

 \section{\bf Homogeneous solenoidal function spaces and the Stokes operator  }

\label{solenoidal}

Set $C^\infty_{0,\sigma}(\Om)=\{u\in [C^\infty_0(\Om)]^n:\  \mbox{\rm div} \,  u=0 \mbox{ in }\Om\}$. 
For  $1\leq p,q\leq \infty$ and  $s\in {\mathbb R}$,  we set 
\begin{align*}
 \dot{H}^s_{p,\si}({\mathbb R}^n)&=\{u\in [\dot{ H}^s_{p} ({\mathbb R}^n)]^n: \mbox{\rm div} \, u=0\},\\
 \dot{B}^s_{p,q,\si}({\mathbb R}^n)&=\{u\in [\dot{ B}^s_{p,q} ({\mathbb R}^n)]^n: \mbox{\rm div}\, u=0\}.
\end{align*}
It is obvious that  $\dot{H}^{0}_{p,\si}({\mathbb R}^n)=L^p_\si({\mathbb R}^n):=\{u\in L^p({\mathbb R}^n) : \mbox{\rm div }u=0\}$.
Here and hereafter, we understand $\|u\|_{X}=\sum_{j=1}^n\|u_j\|_X$ for a vector valued distribution  $u=(u_1,\cdots, u_n)$ whose components are all  in the normed space $X$. 
The following is known in \cite{danchin2}.
\begin{prop}
\label{prop3.1}

Let $1<p<\infty,  1\leq q<\infty$ with  $[ s<\frac{n}{p}]\mbox{ or }[s=\frac{n}{p}\mbox{ and }q=1].$
\begin{itemize}

\item[(i)]
 $C^\infty_{0,\si}({\mathbb R}^n)$ is dense in $\dot{H}^s_{p,\si}({\mathbb R}^n)$ and $\dot{B}^s_{p,q,\si}({\mathbb R}^n)$.

\item[(ii)]
Let $m$ be a positive integer with $m\theta<\frac{n}{p}$. Then, 
\[
(L^p_\si({\mathbb R}^n), \dot{H}^m_{p,\si}({\mathbb R}^n))_{\theta,q}=\dot{B}^{m\theta}_{p,q,\si}({\mathbb R}^n).
\]

\end{itemize}

\end{prop}

 Let   $\Om$ be an exterior domain and  $s>0$.
  We define $ \dot{H}^s_{p,\si}(\Om)$ and $ \dot{B}^s_{p,q,\si}(\Om)$ by the restriction of  $\dot{H}^s_{p,\si}({\mathbb R}^n)$ and $ \dot{B}^s_{p,q,\si}({\mathbb R}^n)$, respectively.  The norms of restriction spaces  are defined similarly as  \eqref{norminOm}.
It is obvious that $C^\infty_{0,\si} (\overline{\Om})$ is dense in $\dot{H}^s_{p,\si}(\Om)$ and $\dot{B}^s_{p,q,\si}(\Om)$ for $1 < p< \infty, \,\, [ s<\frac{n}{p}]\mbox{ or }[s=\frac{n}{p}\mbox{ and }q=1],  \,  1\leq q<\infty$. 

 $ \dot{H}^s_{p,0,\si}(\Om)$ and $ \dot{B}^s_{p,q,0,\si}(\Om)$ are defined by 
\begin{align*}
\dot{H}^s_{p,0,\si}(\Om) & = \{ f \in [\dot H^s_{p,0} (\Om)]^n\, : \,\mbox{\rm div }u=0\mbox{ in }{\Om}\},\\
\dot{B}^s_{p,q,0,\si}(\Om) & = \{ f \in [\dot B^s_{p,q,0} (\Om)]^n\, : \,\mbox{\rm div }u=0\mbox{ in }{\Om}\}
\end{align*}
 normed with
\[
  \|u\|_{ \dot H^s_{p,0,\si} (\Om)}:=\|u\|_{\dot{H}^s_{p,0}(\Om)},\quad \|u\|_{ \dot B^s_{p,q,0,\si} (\Om)}:=\|u\|_{\dot{B}^s_{p,q,0}(\Om)}.
  \]

We define 
 \begin{align*}
 \dot H^{-s}_{p,\si}(\Om):= ( \dot H^{s}_{p',0,\si}(\Om))',\quad    \dot H^{-s}_{p,q,0,\si}(\Om):= ( \dot B^{s}_{p',q',\si}(\Om))',\\
  \dot B^{-s}_{p,q,\si}(\Om):= ( \dot B^{s}_{p',q',0,\si}(\Om))',\quad    \dot B^{-s}_{p,q,0,\si}(\Om):= ( \dot B^{s}_{p',q',\si}(\Om))'.
  \end{align*}
Let $L^p_\si(\Om):=\{u\in [L^p(\Om)]^n: \mbox{\rm div }u=0\mbox{ in }\Om, \ u\cdot \nu=0\mbox{ on }\pa \Om\}$.
Here,  $\nu$ is the unit outward normal vector at the point of $\pa \Om$. According to (ii)  of Lemma \ref{lemma2.4},  it is obvious that
\begin{align}
\label{eq3.3}
\dot{H}^1_{p,0,\si}=\{u\in [L^{\frac{np}{n-p}}(\Om)]^n: \ \nabla u\in [L^p(\Om)]^{n^2}, \mbox{\rm div }u=0\mbox{ in }\Om, \ u|_{\pa \Om}=0\},\  1<p<n.
\end{align}

We understand  the restriction operator $R_\Om $   for a vector valued distribution $f=(f_1,\cdots, f_n)$ by 
  \[
  R_{\Om}f=(R_\Om f_1,\cdots, R_\Om f_n).  
   \]
   Then,   it is trivial that  
$R_{\Om}$ is bounded from    $\{L^p_\si({\mathbb R}^n) , \dot{H}^1_{p,\si}({\mathbb R}^n)\}$ to 
   $\{\dot H^0_{p,\si} (\Om), \dot{H}^1_{p,\si}(\Om)\}$.

We  also understand  the zero extension operator ${\mathcal E}_\Om $ for a vector valued distribution $f=(f_1,\cdots, f_n)$  by 
  \begin{align}\label{zeroextension}
  {\mathcal E}_{\Om}f=({\mathcal E}_\Om f_1,\cdots, {\mathcal E}_\Om f_n). 
   \end{align}
   Then,   it is trivial that  
${\mathcal E}_{\Om}$ is bounded from    $\{L^p_\si(\Om) , \dot{H}^1_{p,0,\si}(\Om)\}$ to 
   $\{L^p_\si({\mathbb R}^n) , \dot{H}^1_{p,\si}({\mathbb R}^n)\}$.
Here we use the same notation $\mathcal E_\Omega$ as in Section~\ref{notation}, now applied componentwise to vector-valued functions.

Modifying the proof of the  interpolation property in  Proposition \ref{Proposition2.8},  the following interpolation property also holds.
\begin{lemm}
\label{lemma3.1}
Let $\Om$ be an exterior domain.  
Then,
\begin{itemize}
\item[(i)]
$C^\infty_{0,\si} (\Om)$ is dense in $L^p_\si(\Om)$ for $1<p<\infty$ and dense in $\dot{H}^1_{p,0,\si}(\Om)$ for $1<p<n.  $

\item[(ii)]
There is extension operator $E_{\Om,\si}: (C_{0,\si}^\infty(\Om))'\rightarrow (C^\infty_{0,\si}({\mathbb R}^n))'$ so that
$E_{\Om,\si}$ is bounded from $\{ \dot H^0_{p, \si} (\Om), \dot{H}^1_{p,\si}(\Om)\}$ to $\{L^p_\si({\mathbb R}^n), \dot{H}^1_{p,\si}({\mathbb R}^n)\}$ for $1<p<n$.

\item[(iii)]

Let $1<p<n, 1\leq q\leq \infty$ and $0<s<1$.
Then,  the following interpolation properties  hold:
\[
[\dot H^0_{p,\si} (\Om), \dot{H}^{1}_{p,\si}(\Om)]_s=\dot{H}^{s}_{p,\si}(\Om),\
(\dot H^0_{p,\si} (\Om), \dot{H}^{1}_{p,\si}(\Om))_{s,q}=\dot{B}^{s}_{p,q,\si}(\Om).\]

\item[(iv)]Let $1<p<n$.
There is bounded operator ${\mathcal R}_{\Om,\si}$ from $\{L^p_\si({\mathbb R}^n), \dot{H}^1_{p,\si}({\mathbb R}^n)\}$ to $\{L^p_\si(\Om), \dot{H}^1_{p,0,\si}(\Om)\}$.

\item[(v)]
Let $1<p<n, 1\leq q<\infty, 0<s<1 $. 
Then, 
\[
[L^p_\si(\Om), \dot{H}^{1}_{p,0,\si}(\Om)]_s=\dot{H}^{s}_{p,0,\si}(\Om),\
(L^p_\si(\Om), \dot{H}^{1}_{p,0,\si}(\Om))_{s,q}=\dot{B}^{s}_{p,q,0,\si}(\Om).
\]
\end{itemize}
\end{lemm}
\begin{proof}
The proof is given in Appendix \ref{appendix2}.
\end{proof}

Having established the fundamental properties of homogeneous solenoidal spaces in exterior domains, 
we are now in a position to investigate the Stokes operator. 
This will provide the analytic framework required for the proof of our main theorem in the next section.

 \begin{rem} 
  By our  definition of  $\dot{B}^{-s}_{p,q,0,\si}(\Om)$  for the exterior domain,  
 it is obvious that   $
   L^n_\si (\Om) \subset \dot B^{-1 +\frac{n}p}_{p,\infty,0,\si  } (\Om), \,\, n < p < \infty$
    with $ \| f\|_{\dot B^{-1 +\frac{n}p}_{p,\infty,0,\si} (\Om)} \leq c \| f\|_{L^n _\si(\Om)}$.

\end{rem}

Let $D(A^\al)$ be the domain of the fractional Stokes operator $A^\al$ in exterior domain, and let  ${D}^{\al}_{p}$ be  the completion of $D(A^{\al})$ with respect to the norm $\|A^{\al}\cdot\|_{L^p(\Om)}$.
W. Borchers and T. Miyakawa \cite{BM1} showed that  $\dot{H}^1_p$  is equivalent to  $D^{\frac{1}{2}}_p$, which reads as follows: 
\begin{prop}[ 
 \cite{BM1}]
\label{prop.grad}
Let $\Om$ be a smooth exterior domain and let $u\in D(A^{\frac{1}{2}})$. 
Then, 
\begin{align*}
 &{\rm (i)}\,\,  \|A^{\frac{1}{2}}u\|_{L^p(\Om)}\leq c\|\nabla u\|_{L^p(\Om)}, \quad  1<p<\infty,\\
&{\rm (ii)} \,\,\|\nabla u\|_{L^p(\Om)}\approx \|A^{\frac{1}{2}}u\|_{L^p(\Om)}, \quad  1<p<n,\\
 &{\rm (iii)} \,\,  {D}^{\frac{1}{2}}_{p}=\dot{H}^{1}_{p,0,\sigma} (\Om), \quad   1 <p<n.
 \end{align*}
\end{prop}

 It is known that   $A$ generates a bounded analytic semigroup  $e^{-At}$  in $L^q_\si$.
In \cite{BM1,iwashita, maremonti-solonnikov} the $L^q-L^r$ time decay estimates are given. For the later use, we introduce the estimates for  the case  $n\geq 3$.  
\begin{prop}[\cite{ iwashita}]
\label{LqLr_stokes}
 Let $f \in L^r_\sigma(\Om)$. 
Then,   it holds that  for  $1\leq r< p\leq \infty$ or $1<r=p\leq \infty$, 
\begin{align*}
\|e^{-At}f\|_{L^p (\Omega)}\leq ct^{-\frac{n}{2}(\frac{1}{r}-\frac{1}{p})}
\|f\|_{L^r(\Omega)}, \quad    t>0,
\end{align*}
and  for  $1\leq r< p\leq n$ or $1<r=p\leq n$, 
\begin{align*}
\|\nabla e^{-At}f\|_{L^p(\Omega)}\leq 
ct^{-\frac{1}{2}-\frac{n}{2}(\frac{1}{r}-\frac{1}{p})}\|f\|_{L^r(\Omega)},\quad t>0.
\end{align*}

\end{prop}

\section{\bf  Proof of Theorem \ref{theorem2} }\label{auxilliaries}
\setcounter{equation}{0}
 
 We are now ready to establish Theorem~\ref{theorem2}, using the estimates obtained above. 
For this purpose, we first record the following auxiliary lemma, which will be used in the proof.
 \begin{lemm}
 \label{lemma.3.1}
 Let $\varphi_0\in C^\infty_{0,\si}(\Om)$,  $0<\al<1$, $1<p\leq n$ and $1\leq q< \infty$.   Then,  for $ 1 < r \leq p$ or $1= r < p$.
     \[
\|  e^{-tA}\varphi_0\|_{\dot B^{\al}_{p,q,0,\si} (\Om )}   \leq 
\left\{\begin{array}{l}\vspace{2mm}
c t^{-\frac{\al}2 -\frac{n}{2}(\frac{1}{r}-\frac{1}{p})} \|\varphi_0\|_{L^{r} (\Om)},\\
 ct^{-\frac{n}{2}(\frac{1}{r}-\frac{1}{p})}\|\varphi_0\|_{\dot{B}^\al_{r,q, 0,\si}(\Om)}, \quad p \neq n.
\end{array}\right.
\]

\end{lemm}

\begin{proof}
From Proposition \ref{LqLr_stokes}, for $ 1 < r \leq p \leq  \infty$ or $1 \leq r < p \leq \infty$,   we have 
 \begin{align}
\label{n_1} \| e^{-tA}{\varphi}_0 \|_{ L^p (\Om)}& \leq c t^{  -\frac{n}{2}(\frac{1}{r}-\frac{1}{p})}  
 \| \varphi_0\|_{L^r (\Om)},
 \end{align}
 and  for $ 1 < r \leq p \leq  n$ or $1 \leq r < p \leq  n$,   we have 
 \begin{align}
 \label{n_2}\| \na e^{-tA}{\varphi}_0 \|_{ L^p (\Om)} &\leq c  t^{-\frac12 -\frac{n}{2}(\frac{1}{r}-\frac{1}{p})} \| \varphi_0\|_{L^r (\Om)}.
\end{align}

On the other hand, from Proposition \ref{prop.grad}, for $ 1 < r \leq p < n$ or $1 \leq r < p < n$, we have 
\begin{align}\label{n_3}
\begin{split}
\| \na e^{-tA}{\varphi}_0\|_{ L^p (\Om)}&\leq c\|A^{\frac{1}{2}}e^{-tA}\varphi_0\|_{L^p(\Om)}
 \\
&=c\|e^{-tA}A^{\frac{1}{2}}\varphi_0\|_{L^p(\Om)}\\
& \leq ct^{-\frac{n}{2}(\frac{1}{r}-\frac{1}{p})}\|A^{\frac{1}{2}}\varphi_0\|_{L^r(\Om)}\\
& \leq ct^{-\frac{n}{2}(\frac{1}{r}-\frac{1}{p})}\|\na \varphi_0\|_{L^r(\Om)}.
\end{split}
\end{align}
Applying  Proposition  \ref{Proposition2.8}  to \eqref{n_1} and \eqref{n_2},  for $0<\al<1$, $ 1 < r \leq p \leq  n$ or $1 \leq r < p \leq  n$, 
 we have 
\begin{align*} 
\|  e^{-tA}{\varphi}_0\|_{\dot B^{\al}_{p,q,0,\si} (\Om )}   \leq c t^{-\frac{\al}2 -\frac{n}{2}(\frac{1}{r}-\frac{1}{p})} \| \varphi_0\|_{L^r (\Om)}.
\end{align*}
We complete the proof of the first quantity of Lemma \ref{lemma.3.1}.

Again, applying Lemma \ref{lemma3.1} to \eqref{n_1} and \eqref{n_3}, for   $0<\al<1$, $ 1 < r \leq p < n$ or $1 \leq r < p < n$,   we have
\begin{align*}
\|e^{-tA}\varphi_0\|_{\dot{B}^\al_{p,q,0,\si}(\Om)}\leq  ct^{-\frac{n}{2}(\frac{1}{r}-\frac{1}{p})}\|\varphi_0\|_{\dot{B}^\al_{r,q, 0,\si}(\Om)}, \quad 1 \leq q < \infty.
\end{align*}
We complete the proof of the second  duality of Lemma \ref{lemma.3.1}.
\end{proof}

{\bf  Proof of Theorem \ref{theorem2} }\\
Now, we prove our theorem  by the duality argument. 
Observe the following  identity
\[
\int_\Om e^{-tA}{u}_0(x) \varphi_0(x) dx=<{u}_0,e^{-tA}\varphi_0>
\mbox{ for }u_0, \varphi_0\in C^\infty_{0,\sigma}(\Om).\]
Here,  $<\cdot,\cdot>$ means a dulaity paring between $\dot{B}^{-\al}_{p,\infty,\sigma}$ and $ \dot{B}^\al_{p',1,0,\sigma}(\Om)$.

Let $ \frac{n}{n -1} < p < \infty$, $ 1  <  q \leq  \infty$ and $ 0 < \al < 1$  so that $1 < p' < n$ and $1 \leq q'< \infty$. 
Let ${u}_0\in \dot{b}^{-1 +\frac{n}p}_{p,q,\sigma}(\Om):=$ the completion of $C^\infty_{0,\si}(\Om)$ in $\dot{B}^{-1 +\frac{n}p}_{p,q,\sigma}$.  From Lemma \ref{lemma.3.1}, for $1<p' <n$, $0<\al<1$ and $1\leq q'<\infty$, we have 
 \[
\|  e^{-tA}\varphi_0\|_{\dot B^{\al}_{p',q',0,\si} ({\mathbb R}^n )}   \leq\left\{\begin{array}{l} \vspace{2mm}
 c t^{-\frac{\al}2 } \|\varphi_0\|_{L^{p'}_\si (\Om)},\\
c\|\varphi_0\|_{\dot{B}^\al_{p',q',0,\si}(\Om)}.
\end{array}\right.
\]

Hence,  we have the inequalities
\begin{align*}
\Big|\int_\Om e^{-tA}{u}_0(x) \varphi_0(x) dx\Big|\leq 
\left\{\begin{array}{l}\vspace{2mm}
ct^{-\frac{\al}{2}}\|{u}_0\|_{\dot{B}^{-\al}_{p,q,\si}(\Om)}\|\varphi_0\|_{L^{p'}_\si (\Om)},\\
 c\|{u}_0\|_{\dot{B}^{-\al}_{p,q,\si}(\Om)}\|\varphi_0\|_{\dot{B}^\al_{p',1,0,\si}(\Om)}.
\end{array}\right.
\end{align*}
These implies  Theorem \ref{theorem2}.

\section{Proof of Theorem \ref{bilinear1}  }
\setcounter{equation}{0}
\label{section.theorem.1.4}

\begin{lemm}
\label{lemma.3.2}
Let $\varphi\in C^\infty_{0,\si}(\Om)$,   $ 0 < \al <1 $, $1<p<n$ and $1\leq q< \infty$.
Then,  it holds that
\begin{align*}
\| \na e^{-tA}{\varphi}_0 \|_{ L^{p} (\Om)} 
\leq 
 ct^{-\frac{1-\al}{2}-\frac{n}{2}(\frac{1}{r}-\frac{1}{p})}\|\varphi_0\|_{\dot{B}^\al_{r,q,0,\si}(\Om)}
\end{align*}
for $   1 <r \leq  p$ or $ 1=r<p.$
\end{lemm}

\begin{proof}
From the second estimate of Proposition \ref{LqLr_stokes}, for $ 1 < r \leq p \leq n$ or $1 \leq r < p \leq n$,   we have 
 \begin{align}
\label{n11} 
 \| \na e^{-tA}{\varphi}_0 \|_{ L^p (\Om)} &\leq c  t^{-\frac12 -\frac{n}{2}(\frac{1}{r}-\frac{1}{p})} \| \varphi_0\|_{L^r (\Om)}.
\end{align}

On the other hand, combining Proposition \ref{prop.grad} and Proposition \ref{LqLr_stokes}, we obtain
\begin{align}
\label{n22}
\begin{split}
\|\nabla e^{-tA}{\varphi}_0\|_{L^p (\Omega)} 
&\leq c\|A^{\frac{1}{2}}e^{-tA}\varphi_0\|_{L^p(\Omega)} \\
&= \|e^{-tA}A^{\frac{1}{2}}\varphi_0\|_{L^p(\Omega)} \\
&\leq ct^{-\frac{n}{2}(\frac{1}{r}-\frac{1}{p})}\|A^{\frac{1}{2}}\varphi_0\|_{L^r(\Omega)} \\
&\leq ct^{-\frac{n}{2}(\frac{1}{r}-\frac{1}{p})}\|\nabla\varphi_0\|_{L^r(\Omega)}
\end{split}
\end{align}
for $1 < r \leq p < n$ or $1 \leq r < p < n$.

 Apply the  interpolation property in Lemma \ref{lemma3.1}  to \eqref{n11} and \eqref{n22}. Then we have
\begin{align*}
\|\nabla e^{-tA}\varphi_0\|_{L^p(\Om)}\leq 
 ct^{-\frac{1-\al}{2}-\frac{n}{2}(\frac{1}{r}-\frac{1}{p})}\|\varphi_0\|_{\dot{B}^\al_{r,q,0,\si}(\Om)},
\end{align*}
where $0<\al<1, 1<p<n, 1\leq q<\infty.$
\end{proof}

{\bf Proof of Theorem \ref{bilinear1} }\\
For  $\varphi\in C^\infty_{0,\sigma}(\Om)$,  the following identity holds 
\begin{align*}
<e^{-tA}{\mathbb P}\mbox{\rm div}{\mathcal F},\varphi> & =-<{\mathcal F}, \nabla e^{-tA}\varphi> \leq \| {\mathcal F} \|_{L^r (\Om)}  \|\nabla e^{-tA}\varphi\|_{L^{r'}(\Om)}.
\end{align*}
According to Proposition \ref{LqLr_stokes}, 
 for $1< p'\leq r'\leq n$ or $1\leq p'<r'\leq n$ it holds that
\[\|\nabla e^{-tA}\varphi\|_{L^{r'}(\Om)}\leq ct^{-\frac{1}{2}-\frac{n}{2}(\frac{1}{p'}-\frac{1}{r'})}\|\varphi\|_{L^{p'}(\Om)}, \quad t>0.
\]
Hence, we have 
\begin{align*}
\Big|<e^{-tA}{\mathbb P}\mbox{\rm div}{\mathcal F},\varphi>\Big|\leq ct^{-\frac{1}{2}-\frac{n}{2}(\frac{1}{p'}-\frac{1}{r'})}\|{\mathcal F}\|_{L^r(\Om)}\|\varphi\|_{L^{p'}(\Om)}.
\end{align*}
This leads to the first estimate of  Theorem \ref{bilinear1}. According to Lemma \ref{lemma.3.2},
\begin{align*}
\| \na e^{-tA}{\varphi}_0 \|_{ L^{r'} (\Om)} &\leq c  t^{-\frac{1-\al}2 -\frac{n}{2}(\frac{1}{p'}-\frac{1}{r'})} \| \varphi_0\|_{\dot B^\al_{p',q',0,\si }  (\Om)}
\end{align*}
for $   1 <p' \leq  r' < n$ or $ 1=p'<r'< n$, $ 0 < \al <1 $.  
Hence, we have 
\begin{align*}
\Big|<e^{-tA}{\mathbb P}\mbox{\rm div}{\mathcal F},\varphi>\Big|\leq ct^{-\frac{1-\al}{2}-\frac{n}{2}(\frac{1}{p'}-\frac{1}{r'})}\|{\mathcal F}\|_{L^r(\Om)}\|\varphi\|_{ \dot{B}^\al_{p',q',0,\si }  (\Om)}.
\end{align*}
This leads to the second estimate of  Theorem \ref{bilinear1}.

\section{\bf Proof of Theorem \ref{theorem1} (Part I: $L^p$ and $L^n$ Estimates)}

\setcounter{equation}{0}
\label{external-force}

We begin with the construction of approximate solutions by iteration. 
Let ${u}^{(0)} = e^{-tA}{u}_0$. After obtaining ${u}^{(1)},\cdots, {u}^{(m)}$, we 
construct  ${u}^{(m+1)}$ defined by
\begin{equation}
\label{approx_n}
{u}^{(m+1)}(t):=e^{-tA}{u}_0-\int^t_0e^{-(t-\tau)A}{\mathbb P}\mbox{\rm div}({u}^{(m)}\otimes {u}^{(m)})(\tau) d\tau.
\end{equation}
 
 Before proceeding, we recall a well-known identity related to the Beta function. This elementary calculation will be frequently used throughout this section to estimate the time singular integrals
 \begin{lemm}
 \label{Lemma2}
  Let $\alpha, \beta < 1$. Then there exists a constant $C_{\alpha, \beta} = B(1-\alpha, 1-\beta)$ such that for all $t > 0$,
  $$\int_0^t (t-\tau)^{-\alpha} \tau^{-\beta} d\tau = C_{\alpha, \beta} \, t^{1-\alpha-\beta}.$$
 \end{lemm}
 
 \subsection{\bf Uniform boundedness }

 \subsubsection*{Step 1a: Estimate in $L^p$  with time decay}
 
 \mbox{}\\
 
 Let  $n < p < \infty$. 
  Let ${u}_0\in \dot{b}^{-1+\frac{n}{p}}_{p,\infty,\si}(\Om)$ and 
 \begin{align*}
   \| {u}_0 \|_{ \dot{B}^{-1 +\frac{n}p}_{p,\infty, \si} } : =  M_0.
   \end{align*}
    According to Theorem \ref{theorem2}, we have   
\begin{align*}
\|e^{-tA}{u}_0\|_{L^p  (\Om)} & \leq  c_1
t^{-\frac{1}{2} +\frac{n}{2p} }  \| {u}_0 \|_{ \dot{B}^{-1 +\frac{n}p}_{p,\infty,\si}  (\Om) }=c_1
t^{-\frac{1}{2} +\frac{n}{2p} } M_0.
\end{align*}
This implies that  \begin{align}\label{0207-4}
\|  {u}^{(0)}\|_{L^\infty_{\frac12 -\frac{n}{2p}} (0, \infty; L^p (\Om))} \leq c_1M_0.
\end{align}

 Assume  that $ \|  {u}^{(k)} \|_{L^\infty_{\frac12 -\frac{n}{2p}} (0, \infty; L^p (\Om))}\leq  M$, $k=1,\cdots,m$.
According to  Theorem   \ref{bilinear1}, we have 
\begin{align}
\label{kind2}
\begin{split}
\int^t_0\|e^{-(t-\tau)A}{\mathbb P}\mbox{\rm div}({u}^{(m)}\otimes {u}^{(m)})(\tau) \|_{L^p(\Om)}d\tau
& \leq c \int^t_0(t-\tau)^{-\frac{1}{2}-\frac{n}{2p}}\|u^{(m)} \otimes u^{(m)}\|_{L^{\frac{p}2}(\Om)}  d\tau\\
&\leq c\int^t_0(t-\tau)^{-\frac{1}{2}-\frac{n}{2p}}\|u^{(m)}\|_{L^p(\Om)}^2 d\tau\\
&\leq c\int^t_0(t-\tau)^{-\frac{1}{2}-\frac{n}{2p}}\tau^{-1+\frac{n}{p}}d\tau \, \|u^{(m)}\|_{L^\infty_{\frac{1}{2}-\frac{n}{2p}}(0,\infty; L^p(\Omega))}^2 \\
& \leq c t^{-\frac{1}{2}+\frac{n}{2p}} M^2, \quad n<p<\infty,
\end{split}
\end{align}
where the last inequality follows directly from Lemma \ref{Lemma2}.

 Combining \eqref{approx_n},  \eqref{0207-4} with \eqref{kind2} we have  
\begin{align*}
\begin{split}
\|{u}^{(m+1)}\|_{L^\infty_{\frac12 -\frac{n}{2p}} (0, \infty;  L^p (\Om)) }
& \leq c_1 M_0 +c_2M^2 .
\end{split}
 \end{align*}

We take $M_0 > 0$ so small that  
\begin{equation}
\label{cond_1}
    c_1 c_2 M_0 \leq \frac{1}{4}
\end{equation}
and define $M$ by 
\begin{equation}
\label{cond_2}
    M = 2 c_1 M_0. 
\end{equation}
Then, for the base case $m=0$, we trivially have $\|u^{(0)}\|_{L^\infty_{\frac{1}{2}-\frac{n}{2p}}(0,\infty; L^p(\Omega))} \leq c_1 M_0 = \frac{M}{2} \leq M$.
Assuming that the bound $\|u^{(k)}\|_{L^\infty_{\frac{1}{2}-\frac{n}{2p}}(0,\infty; L^p(\Omega))} \leq M$ holds for $k=1,\cdots,m$, we substitute this into the recursive inequality to obtain:
\begin{align*}
\begin{split}
\|u^{(m+1)}\|_{L^\infty_{\frac{1}{2}-\frac{n}{2p}} (0, \infty; L^p (\Om)) }
& \leq c_1 M_0 + c_2 M^2 \\
& = c_1 M_0 + c_2 (2 c_1 M_0)^2 \\
& = c_1 M_0 (1 + 4 c_1 c_2 M_0).
\end{split}
\end{align*}
In view of condition \eqref{cond_1} (i.e., $4 c_1 c_2 M_0 \leq 1$), we arrive at:
\begin{align*}
\|u^{(m+1)}\|_{L^\infty_{\frac{1}{2}-\frac{n}{2p}} (0, \infty; L^p (\Om)) }
& \leq c_1 M_0 (1 + 1) = 2 c_1 M_0 = M.
\end{align*}
Therefore, by mathematical induction, we conclude that
\begin{align}\label{0208-1}
\|u^{(m)}\|_{L^\infty_{\frac{1}{2}-\frac{n}{2p}} (0, \infty; L^p (\Om))}\leq M \quad \text{for all integers } m \geq 0.
\end{align}

Thus we have uniform boundedness in $L^p$. 
Next, we turn to the case of $L^n$.

 \subsubsection*{Step 1b: Estimate in $L^n$  with time  decay 
 } 
 \mbox{}\\
 
 Let $ {u}_0 \in  L^n_\si (\Om) $ with 
\[
\|u_0\|_{L^n(\Om)}:=N_0.
\]

 Note that $  L^n_\si  (\Om) \subset  \dot{B}^{-1 +\frac{n}{p_0}}_{p_0, \infty,\si} (\Om)$ for  
any  $ n < p_0 < \infty$. 
Let 
\[
\|u_0\|_{\dot{B}^{-1 +\frac{n}{p_0}}_{p_0, \infty,\si} (\Om)}:=M_0.
\]
From  the result of Step 1, if we choose $M_0$  and $M$ satisfying \eqref{cond_1} and \eqref{cond_2}, respectively, that is,
	   \[ c_1c_2 M_0\leq \frac{1}{4},\quad  M=   2c_1M_0, \]
then 
 \begin{align} \label{0604-1}
\|{u}^{(m)}\|_{L^\infty_{\frac12 -\frac{n}{2p_0}} (0, \infty; L^{p_0} (\Om))}\leq M\quad \mbox{ for all positive integer }\quad m\geq 0.
\end{align}
According to Theorem \ref{theorem2},  we have 
\begin{align}
\label{0207-2-1}
\|  {u}^{(0)}\|_{L^\infty (0, \infty; L^n (\Om))}  = \|e^{-tA}{u}_0\|_{L^\infty (0, \infty; L^n  (\Om) ) } & \leq  c_5
 \| {u}_0 \|_{L^n (\Om ) }=c_5N_0.
 \end{align}
Assume that
\[
 \|{u}^{(k)}\|_{L^\infty  (0, \infty;  L^n (\Om))}\leq N, \quad  k=1,\cdots, m.
\]
Let $  \frac{n}{n-1}<p_1 < n<p_0$ with  $ \frac1{p_1} = \frac1n + \frac1{p_0}$. 
According to the first estimate  of Theorem  \ref{bilinear1},  the estimate in section 4.1 and \eqref{0604-1}, 
 we have 
\begin{align}\label{kind2-1}
\begin{split}
 &\int^{t}_0\|  e^{-(t -\tau)A} \Big( {\mathbb P} \big( {\rm div} \, {u}^{(m)}(\tau)\otimes  {u}^{(m)}(\tau)  \big) \Big) \|_{L^{n}(\Om)} d\tau \\
 &\leq c\int^{t}_0(t-\tau)^{ -\frac{1}{2} -\frac{n}{2{p_0}} }\|{u}^{(m)}(\tau)\otimes  {u}^{(m)}(\tau)\|_{L^{p_1}  (\Om)}d\tau\\
  &\leq c\int^{t}_0(t-\tau)^{ -\frac{1}{2} -\frac{n}{2p_0} }\| {u}^{(m)}(\tau)\|_{ L^{p_0}  (\Om) } \|{u}^{(m)} \|_{L^n  (\Om)}d\tau\\
 &\leq cNM\int^t_0 (t-\tau)^{-\frac{1}{2} -\frac{n}{2p_0} }\tau^{-\frac12  +\frac{n}{2p_0}}   d\tau=c_7NM.
\end{split}
\end{align}
Here, the exact time-integration in the last equality follows directly from Lemma \ref{Lemma2}.
Combining \eqref{0207-2-1} and  \eqref{kind2-1}, we have
\begin{align}
\label{1224-1}
\begin{split}
 \|{u}^{(m+1)}\|_{L^\infty  (0, \infty;  L^n  (\Om)) }\leq c_5N_0+c_7NM.
\end{split}
 \end{align}

We take    $ N  $ with $c_5N_0=\frac{N}{2}$.  
We take $M_0$  so small that  
\begin{equation}
\label{cond_3}
c_1c_7M_0\leq \frac{1}{4}
\end{equation}
together with the restriction \eqref{cond_2}.
Then
we  have  
\begin{align*}
 \|{u}^{(m+1)}\|_{L^\infty  (0, \infty;  L^n  (\Om)) }\leq N.
 \end{align*}
Hence, we conclude that
 \begin{align}\label{0208-2}
\|{u}^{(m)}\|_{L^\infty  (0, \infty; L^n (\Om))}\leq N \quad \mbox{ for all positive  integer } \quad m\geq 0.
\end{align}

\subsection*{Additional Remark}
We can also obtain the weighted estimate of $\nabla u^{(m)}$ which is uniformly bounded with the time-weight $t^{1/2}$ for all $t > 0$.
Observe that
\begin{align}
\label{eqr1}
\|\nabla e^{-tA}u_0\|_{L^n(\Om)}\leq ct^{-\frac{1}{2}}\| u_0\|_{L^n(\Om)}
=c_5' t^{-\frac12 }N_0.
\end{align}
%

On the other hand, from Proposition \ref{LqLr_stokes}, we have 
\begin{align}
\label{eqr2}
\notag& \int^t_0\|\nabla e^{-(t-\tau)A}{\mathbb P}(u^{(m)}\cdot \nabla u^{(m)})\|_{L^n(\Om)}d\tau\\
&\leq c\int^t_0(t-\tau)^{-\frac{1}{2}}\| e^{-\frac{1}{2}(t-\tau)A}{\mathbb P}(u^{(m)}\cdot \nabla u^{(m)})\|_{L^n(\Om)}d\tau
 \\
\notag& \leq c\int^t_0(t-\tau)^{-\frac{1}{2}-\frac{n}{2p_0}}\| u^{(m)}(\tau)\|_{L^{p_0}(\Om)}\| \nabla u^{(m)}(\tau)\|_{L^n(\Om)}d\tau\\
 &\leq 
 c \int^t_0(t-\tau)^{-\frac{1}{2}-\frac{n}{2p_0}} \tau^{ -1 +\frac{n}{2p_0}}  d\tau\\
\notag & \qquad   \times \Big( \sup_{0<\tau<t}\tau^{\frac{1}{2}-\frac{n}{2p_0}}\| u^{(m)}(\tau)\|_{L^{p_0}(\Om)}\Big)\Big(\sup_{0<\tau<t}\tau^{\frac{1}{2}}\| \nabla u^{(m)}(\tau)\|_{L^n(\Om)}\Big)\\
 &\leq 
 c_7't^{-\frac{1}{2}}\Big(\sup_{0<\tau<t}\tau^{\frac{1}{2}-\frac{n}{2p_0}}\| u^{(m)}(\tau)\|_{L^{p_0}(\Om)}\Big)\Big(\sup_{0<\tau<t}\tau^{\frac{1}{2}}\| \nabla u^{(m)}(\tau)\|_{L^n(\Om)}\Big),
\end{align}
where the last inequality follows directly from Lemma \ref{Lemma2}.

Under the induction hypothesis 
$\sup_{0<\tau<t}\tau^{\frac{1}{2}}\| \nabla u^{(m)}(\tau)\|_{L^n(\Om)}\leq \tilde{M}$,  
combining the linear and nonlinear estimates leads to
\begin{equation*}
t^{\frac{1}{2}}\|\nabla u^{(m+1)}(t)\|_{L^n(\Om)} \leq c_5' N_0 + c_7' M_0 \tilde{M}.
\end{equation*}

Hence, if $\tilde{M}=2c_5'N_0$ and $c_1c_7'M_0\leq \frac{1}{4}$ together with the restriction \eqref{cond_2}, then
\[
\sup_{t>0} t^{\frac{1}{2}}\|\nabla u^{(m)}(t)\|_{L^n(\Om)}\leq \tilde{M} \quad \mbox{ for all integers} \quad  m\geq 0.
\]

Therefore, the sequence $\{u^{(m)}\}$ is uniformly bounded in both $L^p$ and $L^n$, (and similarly for their gradients),
provided that the initial data satisfies the smallness conditions:
\begin{equation}
c_1c_2 M_0 \leq \tfrac{1}{4}, \quad 
c_1c_7 M_0 \leq \tfrac{1}{4} \quad 
(c_1c_7' M_0 \leq \tfrac{1}{4} 
\mbox{ for the gradients}).
\end{equation}
\subsection{\bf Uniform convergence}
\mbox{}\\

To prove the convergence of the sequence $\{u^{(m)}\}$, we define $U^{(m)} := u^{(m+1)} - u^{(m)}$ for $m \ge 0$. Then $U^{(m)}$ satisfies the integral equation
\begin{equation}
\label{approx_n2}
{U}^{(m)}(t) = -\int^t_0 e^{-(t-\tau)A}{\mathbb P}\mbox{\rm div}({U}^{(m-1)}\otimes {u}^{(m)}+{u}^{(m-1)}\otimes  {U}^{(m-1)})(\tau) d\tau
\end{equation}
for $m \ge 1$ (with $U^{(0)} := u^{(1)} - u^{(0)}$).

\subsubsection*{Step 2a: Convergence in $L^p$ with time decay} 
\mbox{}\\

Suppose that $\{u^{(m)}\}$ is uniformly bounded in $L^\infty_{\frac{1}{2}-\frac{n}{2p}} (0, \infty; L^p (\Omega))$ satisfying \eqref{0208-1}.
Taking the $L^p$ norm of \eqref{approx_n2} and applying the bilinear estimate along with Hölder's inequality, we explicitly have
\begin{align*}
\|U^{(m)}(t)\|_{L^p(\Om)} 
&\leq c \int_0^t (t-\tau)^{-\frac{1}{2}-\frac{n}{2p}} \Big( \|U^{(m-1)}(\tau)\|_{L^p}\|u^{(m)}(\tau)\|_{L^p} \\
&\qquad\qquad\qquad\qquad\quad + \|u^{(m-1)}(\tau)\|_{L^p}\|U^{(m-1)}(\tau)\|_{L^p} \Big) d\tau \\
&\leq c \int_0^t (t-\tau)^{-\frac{1}{2}-\frac{n}{2p}} \tau^{-1+\frac{n}{p}} d\tau \\
&\qquad \times \Big( \|u^{(m)}\|_{L^\infty_{\frac12 -\frac{n}{2p}}} + \|u^{(m-1)}\|_{L^\infty_{\frac12 -\frac{n}{2p}}} \Big) \|U^{(m-1)}\|_{L^\infty_{\frac12 -\frac{n}{2p}}},
\end{align*}
where the norms in the last line are taken over the space $L^\infty_{\frac12 -\frac{n}{2p}} (0, \infty; L^p (\Om))$. 
Applying Lemma \ref{Lemma2} (or the exact time-integration property) and recalling the uniform bound $M \le c_1 M_0$ from the previous step, we obtain the recurrence relation:
\begin{align}
\begin{split}
\|{U}^{(m)}\|_{L^\infty_{\frac12 -\frac{n}{2p}} (0, \infty; L^p (\Om))}
& \leq c_8 \Big( \|{u}^{(m)}\|_{L^\infty_{\frac12 -\frac{n}{2p}}}
+ \|{u}^{(m-1)}\|_{L^\infty_{\frac12 -\frac{n}{2p}}} \Big)
\|{U}^{(m-1)}\|_{L^\infty_{\frac12 -\frac{n}{2p}}}\\
& \leq c_8 (M + M) \|{U}^{(m-1)}\|_{L^\infty_{\frac12 -\frac{n}{2p}}}\\
& \leq 2c_1c_8 M_0 \|{U}^{(m-1)}\|_{L^\infty_{\frac12 -\frac{n}{2p}}}.
\end{split}
\end{align}

By Assumption~\eqref{cond_2} together with the smallness condition $c_1c_8 M_0 \leq \tfrac{1}{4}$, we deduce
\[
\|{U}^{(m)}\|_{L^\infty_{\frac12 -\frac{n}{2p}} (0, \infty; L^p (\Om))}
\leq \tfrac{1}{2}\|{U}^{(m-1)}\|_{L^\infty_{\frac12 -\frac{n}{2p}} (0, \infty; L^p (\Om))}.
\]

This geometric decay guarantees the absolute convergence of the series $\sum_{k=0}^\infty {U}^{(k)}$ in the Banach space $L^\infty_{\frac{1}{2}-\frac{n}{2p}} (0, \infty; L^p (\Omega))$.
Since the partial sum satisfies the telescoping identity $\sum_{k=0}^{m-1}{U}^{(k)} = {u}^{(m)} - {u}^{(0)}$,
the sequence $\{u^{(m)}\}$ converges strongly to a limit function $u$ in $L^\infty_{\frac{1}{2}-\frac{n}{2p}} (0, \infty; L^p (\Omega))$.
Finally, by the weak lower semi-continuity of the norm, the limit $u$ satisfies the same uniform estimate \eqref{0208-1}.

\subsubsection*{Step 2b: Convergence in $L^n$ with time decay}
\mbox{}\\

Suppose that $\{u^{(m)}\}$ is uniformly bounded in $L^\infty (0, \infty; L^n_\si (\Om))$ satisfying \eqref{0208-1} for some $p_0>n$ and \eqref{0208-2}.
By the exact same process (using Hölder's inequality and Lemma \ref{Lemma2}) as appeared in the proof of \eqref{1224-1}, we obtain the following estimate for the $L^n$ norm:
\begin{align*}
\begin{split}
\|{U}^{(m)}\|_{L^\infty  (0, \infty; L^n (\Om))}
& \leq c_9 \Big( \|{u}^{(m)}\|_{L^\infty_{\frac12 -\frac{n}{2p_0}} (0, \infty; L^{p_0} (\Om))}  \\
& \qquad\qquad + \|{u}^{(m-1)}\|_{L^\infty_{\frac12 -\frac{n}{2p_0}} (0, \infty; L^{p_0} (\Om))} \Big) \\
& \qquad \times \|{U}^{(m-1)}\|_{L^\infty  (0, \infty; L^n (\Om))}\\
& \leq 2c_1c_9 M_0\|{U}^{(m-1)}\|_{L^\infty  (0, \infty; L^n (\Om))}.
\end{split}
\end{align*}

By Assumptions~\eqref{cond_2} and \eqref{cond_3}, together with the smallness condition $c_1c_9 M_0 \leq \tfrac{1}{4}$, we deduce
\[
\|{U}^{(m)}\|_{L^\infty  (0, \infty; L^n (\Om))}\leq  \frac{1}{2}
 \|{U}^{(m-1)}\|_{L^\infty  (0, \infty; L^n (\Om))}.
\]
Following the exact same reasoning as in Step 2a, the sequence $\{{u}^{(m)}\}$ converges strongly to a limit $u$ in  $L^\infty (0, \infty; L^n_\si (\Om))$, and by lower semi-continuity, ${u}$ satisfies the estimate \eqref{0208-2}.

Thus we have established the existence of mild solutions in both $L^p$ and $L^n$ spaces. To complete the proof of Theorem~\ref{theorem1}, it remains to show the corresponding bounds in the critical Besov space and the uniqueness of solutions, which we address in the next section.

\section{\bf Proof of Theorem \ref{theorem1} (Part II: Besov Estimates and Uniqueness)}
\label{section.besov}

\setcounter{equation}{0}

In this section we refine the previous existence result by deriving uniform bounds 
in the critical Besov space and proving uniqueness of mild solutions.

It is easy to check that
\begin{align}\label{240118-1}
{u}=e^{-tA}{u}_0-\int^t_0e^{-(t-\tau)A}{\mathbb P}\mbox{\rm div}({u}\otimes {u})(\tau) d\tau\mbox{ in the sense of distributions}.
\end{align}
From the first estimate  of Theorem \ref{theorem2}, for $ n < p < \infty$,  we have   
\begin{align}\label{0207-22}
\|e^{-tA}{ u}_0\|_{ \dot{B}^{-1 +\frac{n}p}_{p, \infty,\si }  (\Om)} \leq c_1  \| { u}_0\|_ { \dot{B}^{-1 +\frac{n}p}_{p, \infty,\si } (\Om) }= c_1M_0.
\end{align}

According to the second estimate  of Theorem \ref{bilinear1}, for $ n < p < \infty$,    we have 
\begin{align}\label{kind2-1-1}
\begin{split}
 &\int^{t}_0\|  e^{-(t -\tau)A} \Big( {\mathbb P} \big( {\rm div} \, { u}(\tau)\otimes  { u}(\tau)  \big) \Big) \|_{  \dot{B}^{-1 +\frac{n}p}_{p, \infty} (\Om)} d\tau \\
 &\leq c\int^{t}_0(t-\tau)^{  -\frac{n}{p} }\|{ u}(\tau)\otimes  { u}(\tau)\|_{L^{\frac{p}2} (\Om)}d\tau\\
  &\leq c\int^{t}_0(t-\tau)^{  -\frac{n}{p} }\| { u}(\tau)\|^2 _{ L^{p}  (\Om) }  d\tau\\
 &\leq cM^2\int^t_0 (t-\tau)^{ -\frac{n}{p} }\tau^{-1  +\frac{n}{p}}   d\tau\\
 &=c_{10}M^2.
\end{split}
\end{align}
Here,  the last identity follows directly from Lemma \ref{Lemma2}.
Combine the linear and bilinear bounds. Then we obtain the following uniform estimate
\begin{align*}
\begin{split}
\|{ u}\|_{L^\infty  (0, \infty;   \dot{B}^{-1 +\frac{n}p}_{p, \infty} (\Om) ) }\leq c_1M_0+c_{10}M^2.
\end{split}
 \end{align*}

\subsection{Uniqueness of solution}
\mbox{}\\

Next, we turn to the uniqueness of mild solutions in this framework.

Let $u, \tilde{u} \in L^\infty_{\frac{1}{2}-\frac{n}{2p}} (0,\infty; L^p_\sigma (\Omega))$ be mild solutions of \eqref{e1} corresponding to the same initial data $u_0$ (in the form \eqref{240118-1}). 
Setting $U = u - \tilde{u}$, since they share the identical initial data, the linear term vanishes, and $U$ satisfies the integral equation
\[
U(t) = -\int^t_0 e^{-(t-\tau)A}{\mathbb P}\mbox{\rm div}\big( U \otimes u + \tilde{u} \otimes U \big)(\tau) d\tau, \quad t>0.
\]

Assume that 
\[
\sup_{0 < t < \infty} t^{\frac{1}{2}-\frac{n}{2p}} \|u(t)\|_{L^p(\Omega)} = M, \quad 
\sup_{0 < t < \infty} t^{\frac{1}{2}-\frac{n}{2p}} \|\tilde{u}(t)\|_{L^p(\Omega)} = M_1.
\]
Then, by the bilinear estimate and Lemma \ref{Lemma2}, we have
\begin{align*}
\|U(t)\|_{L^p(\Om)} & \leq c \int^t_0 (t-\tau)^{-\frac{1}{2}-\frac{n}{2p}} \Big( \|u(\tau)\|_{L^p(\Om)} + \|\tilde{u}(\tau)\|_{L^p(\Om)} \Big) \|U(\tau)\|_{L^p(\Om)} d\tau \\
& \leq c (M + M_1) \int^t_0 (t-\tau)^{-\frac{1}{2}-\frac{n}{2p}} \tau^{-1 +\frac{n}{p}} \tau^{\frac{1}{2}-\frac{n}{2p}} \|U(\tau)\|_{L^p(\Om)} d\tau \\
& \leq c_* (M + M_1) t^{-\frac{1}{2} + \frac{n}{2p}} \|U\|_{L^\infty_{\frac{1}{2}-\frac{n}{2p}}(0,\infty; L^p(\Omega))},
\end{align*}
where $c_* = c \cdot B(\frac{1}{2} - \frac{n}{2p}, \frac{n}{p})$.

Multiplying both sides by $t^{\frac{1}{2}-\frac{n}{2p}}$ and taking the supremum over $t > 0$, we obtain the weighted estimate for the difference:
\begin{align*}
\|U\|_{L^\infty_{\frac{1}{2}-\frac{n}{2p}} (0, \infty; L^p (\Om))}
\leq c_* (M + M_1) \|U\|_{L^\infty_{\frac{1}{2}-\frac{n}{2p}} (0, \infty; L^p (\Om))}.
\end{align*}

By choosing $M_0$ even smaller if necessary (such that $c_*(M + M_1) < 1$ with $M = 2c_1 M_0$ and $M_1 \le M$), we deduce that 
\[
\|U\|_{L^\infty_{\frac{1}{2}-\frac{n}{2p}} (0, \infty; L^p (\Om))} = 0.
\]
Therefore, $U \equiv 0$, which implies $u \equiv \tilde{u}$. Thus, the mild solution is unique in this class.
%
%

\appendix

\numberwithin{equation}{section}

\numberwithin{theo}{section}



%
%
%

 \section{Proofs of function space properties}
\label{appendix1}


\subsection{Proof of Lemma \ref{lemma2.4}}
\label{app:char}

 \mbox{} \\

\noindent{\bf Proof of  (i).}

Let   $ \Om$  be an   exterior domain and let $ f \in \dot H^1_p (\Om)$ for $p\in (1,n)$.
By the restriction,  there is $F \in \dot H^1_p ({\mathbb R}^n)$ such that  $F|_\Om=f$, $ \|f\|_{\dot{H}^1_p(\Om)}\leq \| F\|_{\dot H^1_p ({\mathbb R}^n)} \leq 2 \| f\|_{\dot H^1_p (\Om)}$.  
According to (iii) of Proposition \ref{prop2.2} in Section \ref{notation},  we have
\begin{align*}
\| \na F \|_{L^p ({\mathbb R}^n)}\approx \|  F\|_{\dot H^1_p ({\mathbb R}^n)}.
\end{align*}
On the other hand, by the argument of the restriction it is trivial that
\[
\|\nabla f\|_{L^p(\Om)}\leq  \|\nabla F\|_{L^p({\mathbb R}^n)}.\]
Combining the above two  estimates, we have
\[
\|\nabla f\|_{L^p(\Om)}\leq c\|  F \|_{\dot H^1_p (\Om)}\leq c \| f\|_{\dot H^1_p (\Om)}.
\]
According to Lemma \ref{lemma2.1} in Section \ref{notation},  the following embedding also holds:
\[
\|f\|_{ L^{\frac{np}{n-p}}(\Om) }\leq c\| f\|_{\dot H^1_p (\Om)}.
\]
 
Conversely,  for $ f \in \{ f \in L^{\frac{np}{n-p}}(\Om) \, | \, \na f \in [L^p ( \Om)]^n \}$, we consider  
Stein's extension $ Ef$ (or
   extension operator $E_\Om$ introduced  later on)  such that  
\begin{align*}
\| \na E_\Om f\|_{L^p ({\mathbb R}^n)} &  \leq c \Big(  \|\na  f\|_{L^p (\Om)} + \| f\|_{L^p (\Om_R)}  \Big)   \leq c \Big(  \|\na  f\|_{L^p (\Om)} + \| f\|_{L^{\frac{np}{n-p}}  (\Om_R)}  \Big),\\
\|  E_\Om f\|_{L^{\frac{np}{n-p}}  ({\mathbb R}^n)} &  \leq c \| f\|_{L^{\frac{np}{n-p}}  (\Om)}, 
\end{align*}
where $\Om_R = \Om \cap B_R$ (see \cite{St}).  The second inequality implies $ E_\Om f \in {\mathcal S}_h'$. 
Hence, we conclude that 
\begin{align}
\dot H^1_p (\Om) =  \{ f \in L^{\frac{np}{n-p}}(\Om) \, | \, \na f \in [L^p ( \Om)]^n \}
\end{align}
with $ \|f \|_{\dot H^1_p (\Om)} \approx \| \na f \|_{L^p (\Om)} + \| f\|_{L^{\frac{np}{n-p}} (\Om)}$.

This completes the proof of (i).

\medskip
\noindent{\bf Proof of (ii).}

Let  $K$ be a compact subset in ${\mathbb R}^n$. Then, according to (iv) of Proposition \ref{prop2.2}  in Section \ref{notation}, we have 
\[
\|u\|_{L^p(K)}\leq c(K)\|u\|_{L^{\frac{np}{n-p}}(K)}\leq c\|\nabla u\|_{L^p({\mathbb R}^n)}, \quad  1\leq p<n.\]   
Hence,   $u\in \dot{H}^1_{p}({\mathbb R}^n)$, $1\leq p<n$ implies 
$u\in H^1_{p,loc}({\mathbb R}^n)$.


Let $u \in \dot{H}^1_{p,0}(\Omega)$ with $1 \leq p < n$. 
Recall that ${\mathcal E}_\Omega$ denotes the zero extension operator defined in Section~\ref{notation}. 
By Lemma~\ref{lemma2.5} in Section \ref{notation}, we have 
${\mathcal E}_\Omega u \in \dot{H}^1_p(\mathbb{R}^n)$ and 
\[
\|{\mathcal E}_\Omega u\|_{\dot{H}^1_p(\mathbb{R}^n)} 
= \|u\|_{\dot{H}^1_{p,0}(\Omega)}.
\]
Hence ${\mathcal E}_\Omega u \in H^1_{p,\mathrm{loc}}(\mathbb{R}^n)$.


By the trace theorem   ${\mathcal E}_\Om u|_{\pa \Om}\in B^{1-\frac{1}{p}}_{p,p}(\pa \Om)$.
On the other hand, since ${\mathcal E}_\Om u=0$ on $\Om^c$, we conclude that ${\mathcal E}_\Om u=0$  on $\pa \Om$, that is, $u|_{\pa \Om}=0$.

 Let $\ep>0$ be given. 
Take  $R>0$ such that
\[
\Big(\int_{\Om \backslash B_{R}}|\nabla  u|^p dx\Big)^{\frac{1}{p}}\leq \frac{\ep}{2}, \, \, \mbox{  that is, } \, \,  \|\nabla  u\|_{L^p ( B_{R}^c ) }\leq \frac{\ep}{2}.
\]
Take cut off function $\eta_R\in C^\infty_0({\mathbb R}^n)$ with $\eta=1$ on $B_{\frac{R}{2}}$ and $\eta=0$ on $B_{R}^c $. Set $u_R=u\eta_R$. 
 Observe that $u_R|_{\pa \Om_R}=0$ a.e. and 
\begin{align*}
\|u_R\|_{L^p(\Om_R)}&\leq c\| u\|_{L^p(\Om_R)}\leq c(R)\|\nabla u\|_{L^p(\Om_R)},\\
\|\nabla u_R\|_{L^p(\Om_R)}&\leq c\|\nabla u\|_{L^p(\Om)}+\frac{c}{R}\|u\|_{L^p(\Om_R)}\leq c\|\nabla u\|_{L^p(\Om_R)}.\end{align*}
Here,   the Poincare's inequality that $\|u\|_{L^p(\Om_R)}\leq c\|\nabla u\|_{L^p(\Om_R)}$ holds since $u|_{\pa \Om}=0$.
Therefore, $u_R\in H^1_{p,0}(\Om_R)$.
It is well known that $C^\infty_0(\Om_R)$ is dense in $H^1_{p,0}(\Om_R)$. Hence, there is $v_R\in C^\infty_0(\Om_R)$ with
\[
\|u_R-v_R\|_{{H}^1_p(\Om_R)}\leq \frac{\ep}{2}, \,\, \mbox{which implies that } \,\, \|\nabla u_R-\nabla v_R\|_{L^p(\Om_R)}\leq \frac{\ep}{2}.
\]
Then, we have
\begin{align*}
\|u-v_R\|_{\dot{H}^1_p({\mathbb R}^n)}&=\|\nabla u-\nabla v_R\|_{L^p(\Om)}\leq \|\nabla u\|_{L^p(\Om_{R}^c)}+\|\nabla u_R-\nabla v_R\|_{L^p(\Om_R )}\leq \epsilon.
\end{align*}
This leads to the conclusion that $C^\infty_0(\Om)$ is dense in $\dot{H}^1_{p,0}(\Om)$.

Also, with (i) of Lemma \ref{lemma2.4} in Section \ref{notation} it holds
\[
\dot{H}^1_{p,0}(\Om)\subseteq \{u\in \dot{H}^1_p(\Om):\ \ u|_{\pa \Om}=0\}.
\]

Let $u\in \dot{H}^1_p(\Om)$ with $u|_{\pa \Om}=0$.  
Observe that \[\int_{{\mathbb R}^n}{\mathcal E}_\Om  u  (x) \partial_{x_l}\phi(x) dx=\int_\Om u(x)\partial_{x_l}\phi(x) dx=-\int_\Om \phi(x) \partial_{x_l} u(x)dx,\] since $u|_{\pa \Om}=0$.
Hence, $\nabla {\mathcal E}_\Om  u\in L^p({\mathbb R}^n)$ with 
 \begin{align*}
\| {\mathcal E}_\Om u\|_{L^{\frac{np}{n-p}}({\mathbb R}^n)}+\|\nabla {\mathcal E}_\Om  u\|_{L^p({\mathbb R}^n)}= \| u\|_{L^{\frac{np}{n-p}}(\Om)}+\|\nabla u\|_{L^p(\Om)}.
 \end{align*}
This implies that
${\mathcal E}_\Om u\in \dot{H}^1_p({\mathbb R}^n)$, that is, $u\in \dot{H}^1_{p,0}(\Om)$.

This completes the proof of (ii).

\medskip

\subsection{Proof of Proposition \ref{Proposition2.7}}
\label{app:oper}

\mbox{} \\
\noindent{\bf Proof of (i).}


 Let $\Om$ be an exterior domain with compact boundary. 

Since $\pa \Om$ is compact set, 
$\partial \Om$ can be covered by the finite number of open balls $\{U_j\}_{k=1}^m,$ where $U_j\cap \pa\Om\neq \emptyset$ are represented by smooth functions  so that there are deffeomorphisms $h_j$ with $h_j(U_j\cap \Om):=V_j^+\subset   B_1^+:=\{x\in {\mathbb R}^n \, | \,  x_n>0,\,  |x|<1\}$, $h_j(U_j\cap \pa\Om)\subset \{(x',0): x'\in {\mathbb R}^{n-1}, \,|x'|<1\}$. 
Without loss of generality, assume that  $\Om^c\subset B_{1}$ and $\Om_{R}\subset \sum_{k=1}^mU_k$  for  $1<R$.
Let $U_0=\overline{\Om_{1}}^c$.

Let  $\{\zeta_j\}_{j=1}^m$ be the partition of unity of $\{U_j\}_{j=1}^m$ so that  $\sum_{j=l}^m\zeta_j=1\mbox{ in }\Om_{R},$ $ \mbox{\rm supp} \, \zeta_j\subset U_j.$
Take a cut-off function $\eta \in C^\infty_0(B_{R})$ with $\eta=1$ on $B_{1}$.
Set $\eta_0=1-\eta$, $\eta_k=\eta \zeta_k$ for $k=1,\cdots, m$.
Then  $\{\eta_j\}_{j=0}^m$ is the partition of unity of $\{U_j\}_{j=0}^m$ with $\sum_{j=0}^m\eta_j=1$ on $\Om$ and $\mbox{\rm supp}\,\eta_j\subset U_j$, $j=0,\cdots,m$.

We define an extension operator $E_\Om:(C^\infty_0(\Om))'\rightarrow (C^\infty_0({\mathbb R}^n))'$ by
\[
E_\Om f(\phi)=f(
R_\Om[\eta_0\phi])+\sum_{j=1}^m f(
R_\Om[\eta_j(\phi-(E_0(\phi\circ h_j^{-1}))\circ h_j])
\]
  for $ \phi\in C^\infty_0({\mathbb R}^n)$ and $ f\in (C^\infty_0(\Om))'$.
Here,  $R_\Om $ is the restriction operator defined in Section \ref{notation} and $E_0$ is  operator   defined by 
\[
E_0\phi(x):=\sum_{k=0}^n a_k \phi(x', -\frac{x_n}{k+1}),  \]
where $ \sum_{k=1}^na_k(-\frac{1}{k+1})^j =1\mbox{ for }j=0,\cdots, m. 
 $
If $f\in L^1_{loc}(\Om)$, then   $E_\Om f $ coincides with
 $E_\Om f=\eta_0 f+\sum_{j=1}^m(E{g}_j)\circ h_j,$
 where
   $g_j=(\eta_j f)\circ h_j^{-1}$, $j=0,\cdots,m$. Here, 
 $E$ is the extension operator from ${\mathbb R}^n_+$ to ${\mathbb R}^n$   defined by
 \begin{align}\label{extension0530}
E{f}(x)=\left\{\begin{array}{l} \vspace{2mm}
f(x), \quad x_n\geq 0,\\
\sum_{k=0}^m(-(k+1)a_k) f(x',-(k+1)x_n)  
:=E_0^*f, \quad x_n<0.
\end{array}\right.
\end{align}
 It is obvious that 
 ${E}_\Om$ is bounded from $ L^p(\Om) $ to $ L^p({\mathbb R}^n)$
 with
\begin{align*}
\|{ E}_\Om f\|_{L^p({\mathbb R}^n)}&\leq c\|f\|_{L^p(\Om)}, \quad  1\leq p\leq \infty.
\end{align*}
Note that $ \dot H^1_p (\Om) = \{ f \in L^{\frac{np}{n-p}} (\Om) \, | \, \nabla f \in [L^p (\Om)]^n \}$ for $ 1 \leq p < n$ from  (ii) of Lemma \ref{lemma2.4}  in Section \ref{notation}.  
Then,   for $   f \in \dot{H}^1_p (\Om), \,\, 1 \leq p  <n$,  we have
\begin{align}\label{0526-3}
\|f\|_{L^p(\Om_R)}\leq   c(R) \|f\|_{L^{\frac{np}{n -p}} (\Om_R)}\leq   c(R)\|f\|_{L^{\frac{np}{n -p}} (\Om)}\leq  
c(R) \| f\|_{\dot H^1_p(\Om)}.
\end{align}
From \eqref{0526-3},  we have 
\[
\|\nabla (E{g}_j)\circ h_j\|_{L^p({\mathbb R}^n)}\leq c\|\nabla f\|_{L^p(U_j\cap \Om)}+c\|f\|_{L^p(\Om_R)}\leq c\| f\|_{\dot H^1_p(\Om)}. \]
Similarly, since $\eta_0f\in C^\infty_0({\mathbb R}^n)$ with $\mbox{\rm supp} \, \eta_0 f\subset U_0$,
\[
\|\nabla(\eta_0 f)\|_{L^p({\mathbb R}^n)}\leq c\|\nabla  f\|_{L^p(U_0)}+c\|f\|_{L^p(\Om_R)}\leq c\| f\|_{\dot H^1_p(\Om)}, \quad 1 \leq  p<n.\]
This leads to the  estimate
\[
\|   E_\Om f\|_{\dot H^1_p({\mathbb R}^n)}\leq c\|   f\|_{\dot H^1_p(\Om)},  \quad 1 \leq p<n.\]
Hence,
${E}_\Om$ is bounded from $\dot{H}^1_p(\Om)$ to $\dot{H}^1_p({\mathbb R}^n)$, $1\leq p<n$.

This completes the proof of (i).

\medskip
\noindent{\bf Proof of (ii).}

We define an operator  ${\mathcal  R}_\Om: f\in (C^\infty_0(\Om))'\rightarrow f\in (C^\infty_0({\mathbb R}^n))' $  by 
\[
{\mathcal R}_\Om f(\phi)=f(R_\Om(\eta_0\phi))+\sum_{j=1}^mf(R_\Om[  \eta_j(\phi-(E_0^*(\phi\circ h_j^{-1}) )\circ h_j  ])
\]
 for $ \phi\in C^\infty_0(\Om)$ and $ f\in (C^\infty_0(\Om))'$. 
If $f\in L^1_{loc}({\mathbb R}^n)$, then ${\mathcal R}_\Om f$ coincides with 
\[
{\mathcal R}_\Om f=\eta_0f+\sum_{j=1}^m(R[g_j-E_0g_j])\circ h_j,
\]
where
   $g_j=(\eta_jf)\circ h_j^{-1}$.

It is obvious that  ${\mathcal R}_\Om$ 
is bounded from $ L^p({\mathbb R}^n) $ to $  L^p(\Om)$. 
 Observe that $\mbox{\rm supp}(R[g_j-E_0g_j])\circ h_j\subset h_j^{-1}({V}_j)$,    $(R[g_j-E_0g_j])\circ h_j|_{\pa \Om}=0$ and  for $ 1 \leq p < n$, 
\begin{align*}
\|  (R[g_j-E_0g_j])\circ h_j\|_{L^{\frac{np}{n-p}} (U_j)} 
& \leq c \| f\|_{L^{\frac{np}{n-p}} ({\mathbb R}^n)}  \leq c \|\nabla f\|_{L^p ({\mathbb R}^n)},\\
\|\nabla (R[g_j-E_0g_j])\circ h_j\|_{L^p(U_j)} & \leq c \Big( \|\nabla f\|_{L^p(\tilde{U}_j)}+\|f\|_{L^p(\Om_R)} \Big) \\
& \leq c \Big(\|\nabla f\|_{L^p({\mathbb R}^n )}+\|f\|_{L^{\frac{np}{n-p}} ({\mathbb R}^n)} \Big)\\
& \leq c\|\nabla f\|_{L^p({\mathbb R}^n)}.
 \end{align*}
On the other hand, $\eta_0f\in C^\infty_0(U_0)$ with 
\begin{align*}
\begin{split}
\| \eta_0 f\|_{L^{\frac{np}{n-p}} (U_0)}& \leq c\|  f\|_{L^{\frac{np}{n-p}} ({\mathbb R}^n) } \leq c\|\nabla f\|_{L^p({\mathbb R}^n)},\\
\|\nabla(\eta_0 f)\|_{L^p(U_0)} & \leq c \Big( \|\nabla  f\|_{L^p(U_0)}+\|f\|_{L^p(\Om_R)} \Big) \leq c\|\nabla f\|_{L^p({\mathbb R}^n)}.
\end{split} 1  <  p<n.
\end{align*}
This leads to the  estimate
\[
\|  {\mathcal  R}_\Om f\|_{\dot H^1_p (\Om)} \leq c\Big(\|\nabla {\mathcal  R}_\Om f\|_{L^p(\Om)} + \| {\mathcal  R}_\Om f\|_{L^{\frac{np}{n-p}} (\Om)} \Big)\leq c\|\nabla  f\|_{L^p({\mathbb R}^n)},  \quad 1 \leq p<n.\]
Hence,  ${\mathcal R}_\Om $ is bounded from $\dot{H}^1_p({\mathbb R}^n)$ to $\dot{H}^1_{p,0}(\Om)$ for $1\leq p<n$.

This completes the proof of (ii).

\medskip

\subsection{Proof of Proposition \ref{Proposition2.8}}
\label{app:interp}

\mbox{} \\
\noindent{\bf Proof of (i).}


Let $f\in \dot{B}^\al_{p,q}(\Om)$, $1<p<n$, $1\leq q\leq \infty$. 
Let $F\in \dot{B}^\al_{p,q}({\mathbb R}^n)$ so that $R_\Om F=f$.  
 According to Lemma \ref{lemma2.3} in Section \ref{notation},  $R_\Om $ is bounded from $\{L^p({\mathbb R}^n), \dot{H}^1_p({\mathbb R}^n)\}$ to  $\{L^p(\Om), \dot{H}^1_p(\Om)\}$.
 Apply  the real  interpolation theorems to  $R_\Om $, then  
   we have  
\begin{align*}
\|f\|_{(L^p(\Om), \dot{H}^1_p(\Om))_{\al,q}}&=\|R_\Om F\|_{(L^p(\Om), \dot{H}^1_p(\Om))_{\al,q}}
\leq c\|F\|_{(L^p({\mathbb R}^n), \dot{H}^1_p({\mathbb R}^n))_{\al,q}}.
\end{align*}
Since $(L^p({\mathbb R}^n), \dot{H}^1_p({\mathbb R}^n))_{\al,q}=\dot{B}^\al_{p,q}({\mathbb R}^n)$(see (vi) of Proposition \ref{prop2.2}  in Section \ref{notation}),  
we have \[
 \|F\|_{(L^p({\mathbb R}^n), \dot{H}^1_p({\mathbb R}^n))_{\al,q}}=\|F\|_{\dot{B}^\al_{p,q}({\mathbb R}^n)}.\]
Combining the above two estimates, we have
 \[
\|f\|_{(L^p(\Om), \dot{H}^1_p(\Om))_{\al,q}}\leq c\|F\|_{\dot{B}^\al_{p,q}({\mathbb R}^n)}.
\]
Taking  infimum over all $F\in \dot{B}^\al_{p,q}({\mathbb R}^n) $ with $F|_\Om=f$,  we have
\[
\|f\|_{(L^p(\Om), \dot{H}^1_p(\Om))_{\al,q}}\leq c\|f\|_{\dot{B}^\al_{p,q}(\Om)}.
\]

 According to (i) of Proposition \ref{Proposition2.7}  in Section \ref{notation}, 
 there is bounded operator $E_\Om$
  from $\{L^p(\Om),\dot{H}^1_p(\Om)\}$\\
   to $\{L^p({\mathbb R}^n),\dot{H}^1_p({\mathbb R}^n)\}$ for $1<p<n$. Observe that $R_\Om E_\Om u=u$  for any $u\in C^m(\overline{\Om})$.
 Since $C^m(\overline{\Om})$ is dense in $\dot{B}^\al_{p,q}(\Om)$ for $1<p<\infty, 1\leq q\leq \infty$,
 $R_\Om E_\Om u=u$  for any $u\in \dot{B}^\al_{p,q}(\Om)$.
 Apply  the interpolation theorems to  $E_\Om$, then  
   we have  
\begin{align*}
\|f\|_{\dot{B}^\al_{p,q}(\Om)}=\|R_\Om E_\Om f\|_{\dot{B}^\al_{p,q}(\Om)}&\leq \|E_\Om f\|_{\dot{B}^\al_{p,q}({\mathbb R}^n)}\\
&\approx\|E_\Om f\|_{(L^p({\mathbb R}^n), \dot{H}^1_p({\mathbb R}^n))_{\al,q}}\leq c\|f\|_{(L^p(\Om), \dot{H}^1_p(\Om))_{\al,q}}.
\end{align*}
Thereore, we conclude that 
\[
(L^p(\Om), \dot{H}^1_p(\Om))_{\al,q}=\dot{B}^\al_{p,q}(\Om), \quad  1<p<n, \quad 1 \leq q \leq  \infty.
\]

This completes the proof of (i).


\medskip
\noindent{\bf Proof of (ii).}


Let $u\in \dot{B}^\al_{p,q,0}(\Om)$ for $1<p<n, 1\leq q<\infty, 0<\al<1$.

Let ${\mathcal E}_\Om$ be the zero exrension operator. Then,  ${\mathcal E}_\Om u=u$ and 
\begin{align*} \|{\mathcal E}_\Om u \|_{\dot{B}^\al_{p,q}({\mathbb R}^n)}=  \|u \|_{\dot{B}^\al_{p,q,0}(\Om) }\quad \mbox{ for } \quad u\in \dot{B}^\al_{p,q,0}(\Om).\end{align*}
Recall that ${\mathcal E}_\Om$ is bounded from $\{L^p(\Om), \dot{H}^1_{p,0}(\Om)\}$ to $\{L^p({\mathbb R}^n), \dot{H}^1_{p}({\mathbb R}^n)\}$(see Lemma \ref{lemma2.5}  in Section \ref{notation}).
By the real interpolation theorem, we have 
\[
\|{\mathcal E}_\Om u\|_{(L^p({\mathbb R}^n), \dot{H}^1_{p}({\mathbb R}^n))_{\al, q}}\leq c\|u\|_{(L^p(\Om), \dot{H}^1_{p,0}(\Om))_{\al, q}}.
\]
Since $(L^p({\mathbb R}^n), \dot{H}^1_{p}({\mathbb R}^n))_{\al, q}=\dot{B}^\al_{p,q}({\mathbb R}^n)$, for $1<p<\frac{n}{\al}$,
\[
\|{\mathcal E}_\Om u\|_{\dot{B}^\al_{p,q}({\mathbb R}^n)}\leq c\|u\|_{(L^p(\Om), \dot{H}^1_{p,0}(\Om))_{\al, q}}.
\] 
Combining the above two estimates,  we have
\[
\|u \|_{\dot{B}^\al_{p,q,0}(\Om)}\leq c\|u\|_{(L^p(\Om), \dot{H}^1_{p,0}(\Om))_{\al, q}}.
\]

Observe that  ${\mathcal R}_\Om {\mathcal E}_\Om u =u$ for any $u\in C^\infty_0(\Om)$.
Since $C^\infty_0(\Om)$ is dense in $\{L^p(\Om), \dot{H}^1_{p,0}(\Om)\}$ for $1<p<n$,  by the real interpolation theorem ${\mathcal R}_\Om {\mathcal E}_\Om f =f$ for any $f\in (L^p(\Om), \dot{H}^1_{p,0}(\Om))_{\al, q}$ for $ 1<p<n, 1\leq q<\infty, \al\in (0,1)$. Hence, we have 
\begin{align*}
\|u\|_{(L^p(\Om), \dot{H}^1_{p,0}(\Om))_{\al,q}}=\|{\mathcal R}_\Om {\mathcal E}_\Om u\|_{(L^p(\Om), \dot{H}^1_{p,0}(\Om))_{\al,q}}&\leq c\|{\mathcal E}_\Om u\|_{(L^p({\mathbb R}^n), \dot{H}^1_p({\mathbb R}^n))_{\al,q}}\\
&\approx \|{\mathcal E}_\Om u \|_{\dot{B}^\al_{p,q}({\mathbb R}^n)}=\|u\|_{\dot{B}^\al_{p,q,0}(\Om)}.
\end{align*}

Therefore,  we conclude that
\[
(L^p(\Om), \dot{H}^1_{p,0}(\Om))_{\al,q}=\dot{B}^\al_{p,q,0}(\Om).
\]

This completes the proof of (ii).

\section{\bf Proof of Lemma \ref{lemma3.1}  }
\label{appendix2}
 
 Let ${\mathcal B}_\Om $ be the Bogovski operator defined by
\begin{align*}
&\mbox{\rm div} \, {\mathcal B}_\Om f=f  \, \quad \mbox{ in } \quad \Om,\quad 
{\mathcal B}_\Om f|_{\pa \Om}=0,\quad  
 \|\nabla {\mathcal B}_\Om f\|_{L^p(\Om)}\leq c\|f\|_{L^p(\Om)}.
\end{align*}
(See  section III.3 of Chapter 3 in \cite{galdi}. For the  bounded domain cases it is  required the condition  $\int_\Om \mbox{\rm div} f dx=0$. Such condition is not necessary for the unbounded domains.)

 \medskip
\noindent{\bf Proof of (i).}

 It is well known that $C^\infty_{0,\si}(\Om)$ is dense in $L^p_\si(\Om)$ for $1<p<\infty$.
   Now,  we show that $C^\infty_{0,\si} (\Om)$ is dense in $\dot{H}^1_{p,0,\si}(\Om)$ for $1<p<n$.
We modify the proof in Appendix \ref{appendix1}, where it was shown that $C^\infty_0(\Om)$ is dense in $\dot{H}^1_{p,0}(\Om)$ for $1<p<n$.

Let $u\in \dot{H}^1_{p,0,\si}(\Om)$, $1<p<n$. 
Let $\ep>0$ be given.  
Take $R>0$  with  $\|\nabla u\|_{L^p(B_{\frac{R}{2}}^c)}<\frac{\ep}{2}$.
Take cut off function $\eta\in C^\infty_0({\mathbb R}^n)$ with $\eta=1$ on $B_{\frac{R}{2}}$ and $\eta=0$ on $B_{R}^c$.
Let $\Om_R=\Om \cap B_R$.  Then   $u_R:=u\eta_R \in [{H}^1_{p,0}(\Om_R)]^n$ with 
\begin{align*}
\|u_R\|_{L^p(\Om_R)}+\|\nabla u_R\|_{L^p(\Om_R)}\leq   c(R)\|\nabla u\|_{L^p(\Om_R)}.\end{align*}
Observe that $\mbox{\rm div}{u_R}=u\cdot \nabla \eta_R$.  Set $v_R=u_R-{\mathcal B}_\Om (u\cdot \nabla \eta_R)$. Then $v_R\in H^1_{p,0,\si}(\Om_R)$
with
\[
\|v_R\|_{L^p(\Om_R)}+\|\nabla v_R\|_{L^p(\Om_R)}\leq   c\|\nabla u\|_{L^p(\Om_R)}.
\]
It is well known that $C^\infty_{0,\si}(\Om_R)$ is dense in $H^1_{p,0,\si}(\Om_R)$. Hence there is $w_R\in C^\infty_{0,\si}(\Om_R)$ with
\[
\|v_R-w_R\|_{{H}^1_p(\Om_R)}\leq \frac{\ep}{2}, \,\, \mbox{ this implies that } \,\, \|\nabla v_R-\nabla w_R\|_{L^p(\Om_R)}\leq \frac{\ep}{2}.
\]
Then,  we have
\begin{align*}
\| u-\ w_R\|_{L^{\frac{np}{n-p}}(\Om)}\leq \|\nabla u-\nabla w_R\|_{L^p(\Om)}\leq \|\nabla u\|_{L^p(\Om_{\frac{R}{2}}^c)}+\|\nabla u_R-\nabla v_R\|_{L^p(\Om{\frac{R}{2}})}\leq \epsilon.
\end{align*}
 Recall\eqref{eq3.3}   in Section \ref{solenoidal} that
 $
\dot{H}^1_{p,0,\si}=\{u\in [L^{\frac{np}{n-p}}(\Om)]^n\, | \,\nabla u\in [L^p(\Om)]^{n^2}, \ \mbox{\rm div }u=0\mbox{ in }\Om, \ u|_{\pa \Om}=0\}, 1<p<n.
$
This leads to the conclusion that $C^\infty_{0,\si}(\Om)$ is dense in $\dot{H}^1_{p,0}(\Om)$.

This completes the proof of (i).
  
\medskip
\noindent{\bf Proof of (ii).}

We define  the operator    ${\bf E}_{\Om}$ 
by
\[{ \bf E}_{\Om}u=({ E}_\Om u_1,\cdots, {E}_\Om u_n)\mbox{ for }u=(u_1,\cdots, u_n),
\]
where 
$E_\Omega$ was introduced in the proof of (i) in Appendix~{app:oper}.

From the property of $E_\Om$, it is obvious that for $1<p<n, 1\leq q<\infty$,  
\[{\bf E}_\Om\mbox{ is also bounded from }\{L^p_\si(\Om), \dot{H}^1_{p,\si}(\Om)\}\mbox{ to }\{[L^p({\mathbb R}^n)]^n, [\dot{H}^1_{p}({\mathbb R}^n)]^n\}.
\]

Define a projection  operator $P_0 $ by $P_0 f=f-\nabla N*\mbox{\rm div}f$ for  $f=(f_1,\cdots, f_n)$, where $N$ is Newtonian potential. 
 Then $P_0$ is bounded  from $ \{[L^p({\mathbb R}^n)]^n,[\dot{H}^1_p({\mathbb R}^n)]^n\}$ to 
$\{L^p_\si({\mathbb R}^n), \dot{H}^1_{p,\si}({\mathbb R}^n)\}$  
for $1<p<n, 1\leq q<\infty$.
Define $ E_{\Om,\si}=P_0{\bf E}_\Om$. Then $E_{\Om,\si}$ is bounded from  $\{\dot H^0_{p, \si } (\Om), \dot{H}^1_{p,\si}(\Om)\}$ to $\{L^p_\si({\mathbb R}^n), \dot{H}^1_{p,\si}({\mathbb R}^n)\}$ for $1<p<n, 1\leq q<\infty$.

This completes the proof of (ii).

\medskip
\noindent{\bf Proof of (iii).}

Recall that  $E_{\Om,\si}$  is bounded from $\{ \dot H^0_{p, \si} (\Om),\dot{H}^1_{p,\si}(\Om)\}$ to $\{L^p_\si({\mathbb R}^n),\dot{H}^1_{p,\si}({\mathbb R}^n)\}$ for $1<p<n$ 
 and that   
  $R_\Om $ is  bounded from   
$\{L^p_\si({\mathbb R}^n), \dot{H}^s_{p,\si}({\mathbb R}^n)
\}$ to $  \{\dot H^0_{p,\si} (\Om), \dot{H}^s_{p,\si}(\Om)
\}$.
Observe that $R_\Om E_{\Om,\si} u=u$  for any $u\in C^\infty_{0,\si}(\overline{\Om})$.

Now we follow the same line of argument as in the proof of (i) in Appendix~\ref{app:interp}. 
Then we obtain the following interpolation 
\[
(\dot H^0_{p,\si} (\Om), \dot{H}^1_{p,\si}(\Om))_{\al,q}=\dot{B}^\al_{p,q,\si}(\Om), \quad  1<p<n, 1 \leq q \leq  \infty.
\]

This completes the proof of (iii).

\medskip
\noindent{\bf Proof of (iv).}

We define  the operator ${{ \Re}}_\Om$ by 
\[
{\Re}_\Om u:=({\mathcal R}_\Om u_1,\cdots, {\mathcal R}_\Om u_n)\mbox{ for }u=(u_1,\cdots, u_n),
\]
where 
 ${\mathcal R}_\Omega$ denotes the scalar operator introduced in the proof of (ii) in Appendix~\ref{app:oper}.

By the property of ${\mathcal R}_\Om$,   it is obvious that  for $1<p<n, 1\leq q<\infty$
\[
{\Re}_\Om \mbox{ is also bounded   from } \{L^p_\si({\mathbb R}^n), \dot{H}^1_{p,\si}({\mathbb R}^n)\} \mbox{ to } \{[L^p(\Om)]^n, [\dot{H}^1_{p,0}(\Om)]^n\}.\]

Let ${\mathcal B}_\Om$ be the Bogovski operator defined at the beginning of this section. 
Define a projection operator $P_\Om :  \{[L^p(\Om)]^n, [\dot{H}^1_{p,0}(\Om)]^n\}\rightarrow  \{L^p_\si(\Om), \dot{H}^1_{p,0,\si}(\Om)\}$ by $
P_\Om f=f-{\mathcal B}_\Om f.$
Now  define ${\mathcal R}_{\Om,\si}=P_\Om{\Re}_\Om$. Then
${\mathcal R}_{\Om,\si}$ is a bounded linear operator from  $\{L^p_\si({\mathbb R}^n), \dot{H}^1_{p,\si}({\mathbb R}^n)\}$ to  $ \{L^p_\si({\mathbb R}^n), \dot{H}^1_{p,0,\si}({\mathbb R}^n)\}$  for $1<p<n, 1\leq q<\infty$.

This completes the proof of (iv).
  
\medskip
\noindent{\bf Proof of (v).}
 
 Recall that  ${\mathcal E}_{\Om}$  is    bounded from $\{L^p_\si(\Om), \dot{H}^s_{p,0,\si}(\Om)
 \}$ to $\{L^p_\si({\mathbb R}^n), \dot{H}^1_{p,0,\si}({\mathbb R}^n)
 \}$ for $1<p<n, 1\leq q<\infty$ 
and that
${\mathcal R}_{\Om,\si}$ is bounded  from  $\{L^p_\si({\mathbb R}^n), \dot{H}^1_{p,\si}({\mathbb R}^n)\}$ to  $ \{L^p_\si({\mathbb R}^n), \dot{H}^1_{p,0,\si}({\mathbb R}^n)\}$  for $1<p<n, 1\leq q<\infty$.
Observe that   ${\mathcal R}_{\Om,\si} {\mathcal E}_\Om f=f$  for all $f\in C^\infty_{0,\si}(\Om)$.

Again we follow the same line of argument as in (ii) in Appendix~\ref{app:interp}. 
It follows that the following interpolation property holds:
\[
(L^p_\si(\Om), \dot{H}^1_{p,0,\si}(\Om))_{\al,q}=\dot{B}^\al_{p,q,0,\si}(\Om), \quad  1<p<n, 1 \leq q \leq  \infty.
\]

This completes the proof of (v).

{\bf Acknowledgement}
Chang (RS-2023-00244630) and Jin (RS-2023-00280597)
are supported by the Basic Research Program
through the National Research Foundation of Korea
funded by Ministry of Science and ICT.

%
%
%

\end{document}